\documentclass{article}
\usepackage{amsthm,amsmath,amsfonts,amssymb, authblk, hyperref,tikz,xcolor}
\usepackage[capitalise]{cleveref}
\usepackage{vmargin}
\setmarginsrb{2cm}{1cm}{2cm}{3cm}{1cm}{1cm}{2cm}{2cm}
\everymath{\displaystyle}
\numberwithin{equation}{section}

\newcommand{\ensemblenombre}[1]{\mathbb{#1}}
\newcommand{\N}{\ensemblenombre{N}}

\newcommand{\R}{} 
\renewcommand{\R}{\ensemblenombre{R}}

\renewcommand{\leq}{\leqslant}
\renewcommand{\geq}{\geqslant}

\newcommand{\abs}[1]{\left\lvert#1\right\rvert}
\newcommand{\norme}[1]{\left\| #1\right\|}

\newcommand{\pare}[1]{\left( #1 \right)}
\renewcommand{\leq}{\leqslant}

\renewcommand{\geq}{\geqslant}

\theoremstyle{plain} 
\newtheorem{Proposition}{Proposition}[section] 
\newtheorem{Theorem}[Proposition]{Theorem}
\newtheorem{Lemma}[Proposition]{Lemma}
\newtheorem{Corollary}[Proposition]{Corollary}
\theoremstyle{definition}
\newtheorem{Definition}[Proposition]{Definition}
\newtheorem{Remark}[Proposition]{Remark}

\begin{document}

\title{Null-controllability of cascade reaction-diffusion systems with odd coupling terms}
\author[1]{Kévin Le Balc'h}
\author[2]{Takéo Takahashi}
\affil[1]{Sorbonne Université, CNRS, Inria, Laboratoire Jacques-Louis Lions, Paris, France}
\affil[2]{Universit\'e de Lorraine, CNRS, Inria, IECL, F-54000 Nancy, France}

\maketitle
\begin{abstract}
In this paper, we consider a nonlinear system of two parabolic equations, with a distributed control in the first equation and an odd coupling term in the second one. 
We prove that the nonlinear system is small-time locally null-controllable. The main difficulty is that the linearized system is not null-controllable. To overcome this obstacle, we extend in a nonlinear setting the strategy introduced in \cite{KLB} that consists in constructing odd controls for the linear heat equation. The proof relies on three main steps. First, we obtain from the classical $L^2$ parabolic Carleman estimate, conjugated with maximal regularity results, a weighted $L^p$ observability inequality for the 
nonhomogeneous heat equation. Secondly, we perform a duality argument, close to the well-known Hilbert Uniqueness Method in a reflexive Banach setting, to prove that the heat equation perturbed by a source term is null-controllable thanks to odd controls. Finally, the nonlinearity is handled with a Schauder fixed-point argument.
\end{abstract}

\vspace{1cm}



\tableofcontents

\section{Introduction}
Let $T>0$ be a positive time, $d \in \mathbb{N}^{*}$, $\Omega$ be a bounded, connected, open subset of $\mathbb{R}^{d}$ of class $C^2$
corresponding to the spatial domain and 
$\omega$ be a nonempty open subset such that $\overline{\omega}\subset \Omega$. In what follows, we use the notation $1_{\omega}$ for the characteristic function of $\omega$.

The null-controllability of the heat equation described below was first obtained by
Fattorini and Russell \cite{MR335014} for $d=1$ and by Lebeau, Robbiano \cite{LR95} and Fursikov, Imanuvilov \cite{FI}  for $d\geq 1$.
More precisely for any $y_0 \in L^2(\Omega)$, there exists $h \in L^{2}((0,T)\times\omega)$ such that the solution $y$ of the system
\begin{equation}
\label{eq:heatcontrol}
\left\{
\begin{array}{ll}
\partial_t y -  \Delta y =h 1_{\omega} &\mathrm{in}\ (0,T)\times\Omega,\\
y= 0&\mathrm{on}\ (0,T)\times\partial\Omega,\\
y(0,\cdot)=y_0& \mathrm{in}\ \Omega,
\end{array}
\right.
\end{equation}
satisfies $y(T,\cdot) = 0$. These results were then extended to a large number of other parabolic systems, linear or nonlinear. For instance, 
the null-controllability of linear coupled parabolic systems has been a challenging issue for the control community in the last two decades. In that direction, we can quote,
among the large literature devoted to this problem, \cite{AKBDGB09}, where Ammar-Khodja, Benabadallah, Dupaix, Gonzalez-Burgos exhibit sharp conditions for the null-controllability of systems of the form
\begin{equation}
\label{eq:Paraboliccontrol}
\left\{
\begin{array}{ll}
\partial_t Y -  D \Delta Y = A Y + B h 1_{\omega} &\mathrm{in}\ (0,T)\times\Omega,\\
Y= 0&\mathrm{on}\ (0,T)\times\partial\Omega,\\
Y(0,\cdot)=Y_0& \mathrm{in}\ \Omega.
\end{array}
\right.
\end{equation}
Here, at time $t \in (0,T]$, $Y(t,.): \Omega \rightarrow \R^n$ is the state, $h=h(t,.) : \Omega \rightarrow \R^m$ is the control, $D:=\operatorname{diag}(d_1,\dots,d_n)$ with $d_i \in (0,+\infty)$ is the \textit{diffusion matrix}, $A \in \R^{n \times n}$ is the \textit{coupling matrix} and $B \in \R^{n \times m}$ represents the \textit{distribution of controls}.  
One objective is to reduce the number of controls $m$ (and in particular to have $m <n$) by using the coupling matrices $A$ and $B$. Let us also quote
the survey \cite{AKBGBdT11} for other results and open problems in that direction.

In this article, we consider the following controlled semi-linear reaction-diffusion system
\begin{equation}
\label{eq:ReactionDiffusionSL}
\left\{
\begin{array}{ll}
\partial_t y_1 - d_1 \Delta y_1 =a_{11 }y_1^{N_1} + h 1_{\omega} &\mathrm{in}\ (0,T)\times\Omega,\\
\partial_t y_2 - d_2 \Delta y_2 =a_{21} y_1^{N_2}+a_{22} y_2^{N_3} &\mathrm{in}\ (0,T)\times\Omega,\\
y_1=y_2= 0&\mathrm{on}\ (0,T)\times\partial\Omega,\\
y_1(0,\cdot)=y_{1,0}, \quad y_2(0,\cdot)=y_{2,0}& \mathrm{in}\ \Omega,
\end{array}
\right.
\end{equation}
where $d_1,d_2 \in (0, +\infty)$, $N_1,N_2,N_3\in \mathbb{N}^*$ and $a_{ij}\in \mathbb{R}$. In \eqref{eq:ReactionDiffusionSL}, at time $t \in [0,T]$, $(y_1,y_2)(t,\cdot) : \Omega \to \R^2$ is the state while $h(t,\cdot) : \omega \to \R$ is the control. We are interested in the null-controllability of \eqref{eq:ReactionDiffusionSL}, that is find a control $h=h(t,x)$, supported in $(0,T)\times \omega$, that steers the state $(y_1,y_2)$ to zero at time $T$, i.e. $(y_1,y_2)(T,\cdot) = 0$. Note that \eqref{eq:ReactionDiffusionSL} is a so-called “cascade system” because the first equation is decoupled from the second equation. For such a system, the basic idea is to use the \textit{nonlinear coupling} term  $a_{21} y_1^{N_2}$, as an indirect control term, that acts on the second component $y_2$.

\subsection{Main results}
Our control results on \eqref{eq:ReactionDiffusionSL} are written in the framework of weak solutions. More precisely, we define the Banach space
\begin{equation}\label{21:47}
\mathcal{W}:=L^2(0,T;H_0^1(\Omega)) \cap H^1(0,T;H^{-1}(\Omega)) \cap L^{\infty}((0,T)\times \Omega),
\end{equation}
and we consider solutions of \eqref{eq:ReactionDiffusionSL} such that $y_1,y_2\in \mathcal{W}$. The precise definition of the weak solutions of 
\eqref{eq:ReactionDiffusionSL} is given in \cref{D2} and a corresponding well-posedness result is stated in
\cref{Th:WPLinftySLReactionDiffusion} for controls $h \in  L^{p}((0,T)\times \omega)$ with 
\begin{equation}\label{09:42}
p\in \left(\frac{d+2}{2}, \infty\right], \quad \text{and} \quad p\geq 2 \ \text{if} \ d=1.
\end{equation}

%
%
Our first main result can be stated as follows.
\begin{Theorem}
\label{th:Mainresult1}
Let $p$ satisfies \eqref{09:42}. Assume
\begin{equation}
\label{eq:CNScontrollability}
a_{2,1} \neq 0, \quad N_2 \ \text{is odd}.
\end{equation}
Then there exists $\delta >0$ such that for any initial data satisfying
\begin{equation}
\label{eq:smalldatasystem}
\norme{y_{1,0}}_{L^{\infty}(\Omega)} + \norme{y_{2,0}}_{L^{\infty}(\Omega)} \leq \delta,
\end{equation}
there exists a control $h \in  L^{p}((0,T)\times \omega)$ satisfying 
\begin{equation}
\label{eq:smallcontrolsystem}
\norme{h}_{L^p((0,T)\times \omega)} \lesssim \delta,
\end{equation}
such that the solution $(y_1,y_2) \in \mathcal{W}\times \mathcal{W}$ of \eqref{eq:ReactionDiffusionSL} satisfies
\begin{equation}
\label{eq:NullStateu1u2}
(y_1,y_2)(T,\cdot) = 0.
\end{equation}
\end{Theorem}
Here and in all that follows, we use the notation $X\lesssim Y$ if there exists a constant $C>0$ such that we have the inequality $X\leq CY$. 
In the whole paper, we use $C$ as a generic positive constant that does not depend on the other terms of the inequality.
The value of the constant $C$ may change from one appearance to another. Our constants may depend on the geometry ($\Omega$, $\omega$), on the time $T$ and on the dimension $d$. If we want to emphasize the dependence on a quantity $k$, we write  $X\lesssim_k Y$.

As we will see, the smallness conditions on the initial data i.e. \eqref{eq:smalldatasystem} and on the control i.e. \eqref{eq:smallcontrolsystem} are sufficient conditions to guarantee the well-posedness of the system \eqref{eq:ReactionDiffusionSL}, see \Cref{Th:WPLinftySLReactionDiffusion} below.

Before continuing, let us make some comments related to \Cref{th:Mainresult1}.
\begin{itemize}
\item The sufficient condition \eqref{eq:CNScontrollability} ensuring the local null-controllability of \eqref{eq:ReactionDiffusionSL} is actually necessary. 
Indeed, if $a_{21} = 0$ then the second equation of \eqref{eq:ReactionDiffusionSL} is decoupled from the first equation so $y_2$ cannot be driven to $0$ at time $T$. Moreover, if $N_2$ is even, the strong maximum principle shows that we can not control $y_2$: assume for instance that $a_{21} \geq 0$, then
\begin{equation}
\label{eq:maxprinciplesecond}
\partial_t y_2 - d_2 \Delta y_2 - a_{22} y_2^{N_3} = a_{21} y_1^{N_2} \geq 0\ \text{in}\ (0,T)\times \Omega
\end{equation} 
and thus $\widetilde{y}_2(t,x):=y_2(t,x)e^{-\lambda t}$, with $\lambda\geq |a_{22}| \left\| y_2 \right\|_{L^{\infty}((0,T)\times \Omega)}^{N_3-1}$ satisfies 
$$
\partial_t \widetilde{y}_2 - d_2 \Delta \widetilde{y}_2+c\widetilde{y}_2 \geq 0\ \text{in}\ (0,T)\times \Omega
$$
with $c\geq 0$ and we can apply the standard strong maximum principle (see, for instance, \cite[Theorem 12, p.397]{Eva10}): if 
$y_{2,0} \geq 0$ and $y_{2,0} \neq 0$ then for all $t \in (0,T]$, $y_{2}(t, \cdot)  > 0$ in $\Omega$.
\item The linear case $$N_1 = N_2 = N_3 = 1,$$ is already treated in \cite{deT00} by de Teresa. To obtain such a result, the author shows a Carleman estimate and deduce from it an observability inequality for the adjoint system.
\item For the semi-linear case, the main idea is to linearize the system in order to use the previous result. However, 
if $N_2 \geq 2$, in the linearized system around the trajectory $((\overline{y_1}, \overline{y_2}), \overline{h}) = ((0,0),0)$, we can see that
the second equation 
is decoupled from the first one and thus can not be controlled; the linearized system is thus not null-controllable.

\item To overcome this difficulty, Coron, Guerrero, Rosier \cite{CGR10} use the return method in the case 
$$
N_2 = 3, \quad N_3=1.
$$ 
More precisely, they construct a reference trajectory $((\overline{y_1}, \overline{y_2}), \overline{h})$ of \eqref{eq:ReactionDiffusionSL} starting from $(\overline{y_1}, \overline{y_2})(0,\cdot) = 0$, reaching $(\overline{y_1}, \overline{y_2})(T,\cdot) = 0$ and satisfying $ |\overline{y_1}| \geq \varepsilon > 0$ in $(t_1,t_2)\times \omega$. Then they linearize \eqref{eq:ReactionDiffusionSL} around the reference trajectory and obtain for the second equation
\begin{equation}
\partial_t y_2 - d_2 \Delta y_2 = 3 a_{21} \overline{y_1}^2 y_1 +  a_{22}  y_2\ \text{in}\ (0,T)\times \Omega.
\end{equation}
They can then use \cite{deT00} to obtain that the null-controllability of the linearized system and then the local
null-controllability of the nonlinear system \eqref{eq:ReactionDiffusionSL} by a fixed-point argument.

\item In \cite{KLB}, the first author employs a new direct strategy in order to deal with the case 
$$
N_2 \ \text{odd and} \  N_3=1,
$$
that we adapt here to prove \Cref{th:Mainresult1} for the more general case
$$
N_2 \ \text{odd and} \  N_3\geq 1.
$$
We describe below the idea of the proof.
\end{itemize}

\textbf{Strategy of the proof.} 
We proceed in two steps: in the first step, we control the first equation of \eqref{eq:ReactionDiffusionSL} in the time interval $(0,T/2)$. Using the small-time 
local null-controllability of the semi-linear heat equation, there exists a control $h$ such that $y_1(T/2,\cdot)=0$. Using the smallness assumptions, we can ensure that the second equation of \eqref{eq:ReactionDiffusionSL} admits a solution on $(0,T/2)$. In the second step, we control this second equation 
thanks to a \textit{fictitious odd control} $H$. More precisely, we can consider the control problem
\begin{equation}
\label{eq:heatSL2}
\left\{
\begin{array}{rl}
\partial_t y_2 - \Delta y_2=H \chi_{\omega}+ y_2^{N_3} &\mathrm{in}\ (T/2,T)\times\Omega,\\
y_2= 0&\mathrm{on}\ (T/2,T)\times\partial\Omega,
\end{array}
\right.
\end{equation}
where $\chi_{\omega} = \widetilde{\chi}_{\omega}^{N_2}$ and where $\widetilde{\chi}_{\omega}\in C^\infty(\Omega)$ has a compact support in $\omega$, 
$\widetilde{\chi}_\omega \not\equiv 0$. We 
then need a control $H$ such that $y_2(T,\cdot)=0$, satisfying $H(T/2,\cdot) = H(T,\cdot)=0$ and such that $H^{1/N_2}$ is regular. Such a control is given by our second main result (\cref{th:mainresult2}) stated below. We can then set in $(T/2,T)$
$$
y_1:=\left(H\chi_{\omega}\right)^{1/N_2}, \quad h:=\partial_t y_1-c_1 \Delta y_1-y_1^{N_1}.
$$
By construction, $((y_1,y_2), h)$ is a trajectory of \eqref{eq:ReactionDiffusionSL} satisfying \eqref{eq:NullStateu1u2}.

To simplify the work and without loss of generality, we assume in what follows that 
$$
d_1=d_2=1, \quad
a_{11}=a_{21}=a_{22}=1, 
\quad 
N_1,N_2,N_3\geq 2.
$$

The proof of \Cref{th:Mainresult1} crucially relies on the construction of odd controls for the semi-linear heat equation that we present now. For $N \geq 2$, we thus consider the system
\begin{equation}
\label{eq:heatSL}
\left\{
\begin{array}{ll}
\partial_t y - \Delta y=h \chi_{\omega}+ y^{N} &\mathrm{in}\ (0,T)\times\Omega,\\
y= 0&\mathrm{on}\ (0,T)\times\partial\Omega,\\
y(0,\cdot)=y_{0}& \mathrm{in}\ \Omega.
\end{array}
\right.
\end{equation}
The definition of the weak solutions for the above system and a corresponding well-posedness result are given in \cref{D1}
and \cref{Th:WPLinftySL}.
Our second main result states as follows.
\begin{Theorem}
\label{th:mainresult2}
Assume that $N \geq 2$, $n \in \N$, and $p \geq 1$. There exists $\delta >0$ such that for every initial data $y_0 \in L^{\infty}(\Omega)$ such that
\begin{equation}
\label{eq:smally0}
\norme{y_0}_{L^{\infty}(\Omega)} \leq \delta,
\end{equation}
there exists a control $h \in L^{\infty}((0,T)\times \omega)$ satisfying
\begin{equation}
\label{eq:smallcontrolheatSL}
\norme{h}_{L^{\infty}((0,T)\times \omega)} \lesssim \norme{y_0}_{L^{\infty}(\Omega)},
\end{equation}
\begin{equation}
\label{eq:oddcontrol}
h^{1/(2n+1)}\in L^p(0,T;W^{2,p}(\Omega)) \cap W^{1,p}(0,T;L^p(\Omega)), \quad
h^{1/(2n+1)}(0,\cdot)=h^{1/(2n+1)}(T,\cdot)=0,
\end{equation}
and such that the solution $y \in \mathcal{W}$ of \eqref{eq:heatSL} satisfies
\begin{equation}
\label{eq:smallyheatSL}
\norme{y}_{\mathcal{W}} \lesssim \norme{y_0}_{L^{\infty}(\Omega)},
\end{equation}
and
\begin{equation}
\label{eq:nully}
y(T,\cdot) = 0.
\end{equation}
\end{Theorem}
As for \Cref{th:Mainresult1}, the smallness conditions \eqref{eq:smally0} and \eqref{eq:smallcontrolheatSL} are sufficient to guarantee the well-posedness of the semi-linear heat equation \eqref{eq:heatSL}, see \Cref{Th:WPLinftySL} below.

Before continuing, let us make some comments related to \Cref{th:mainresult2}.
\begin{itemize}
\item The crucial property in \Cref{th:mainresult2} is the odd behavior of the control, stated in \eqref{eq:oddcontrol}. Actually, the small-time local null-controllability of  \eqref{eq:heatSL} with controls in $L^{\infty}((0,T)\times \omega)$ is a consequence of \cite[Lemma 6]{AT02}.
\item For $N=1$, that is the linear case, the result of \Cref{th:mainresult2} is still true and has already been established by the first author, see \cite[Proposition 3.7]{KLB}. One can even obtain a (small-time) global null-controllability result with odd controls due to the linear setting.
Note that here, we extend the result of \cite{KLB} in the case of a linear heat equation with a source term, see \cref{sec-st}. 
\end{itemize}

\textbf{Strategy of the proof.} 
First, we use a classical Carleman estimate for the nonhomogeneous heat equation to obtain a
weighted $L^2$ observability inequality stated in \Cref{C1}. 
From this result and after that, we need to take care about the weights appearing in the norm of the adjoint system
they have to be ``comparable''. We then deduce from this result a weighted $L^p$ observability inequality, see \Cref{L4Obs} below with an arbitrary large $p$. 
As a consequence, a null-controllability result is obtained for the heat equation with a source term and with odd controls. 
Let us remark that taking $p$ large enough allows us to do only one bootstrap argument for getting the desired odd behavior for the control, see \Cref{P1} below. 
This is different from \cite[Theorem 4.4 and Proposition 4.9]{KLB} where two such arguments are used for obtaining the null-controllability of the heat equation with odd controls. 
Another bootstrap argument is then required in order to deal with the nonlinearity in the fixed-point argument, see \Cref{P7} below. Finally, a Schauder fixed-point argument, see \Cref{Sec:schauder}, is performed to obtain \cref{th:mainresult2}. 
We can remark that here due to our method for constructing the control, in this fixed point argument, the corresponding nonlinear mapping
is $\alpha$-H\"older continuous with $\alpha<1$. In particular, a Banach fixed point argument does not seem to apply.

\subsection{Outline of the paper}
The outline of the paper is as follows. In \Cref{sec:WPparabolicresults}, we recall some standard facts about well-posedness, regularity results for linear and nonlinear heat equations in various functional settings. We notably prove that \eqref{eq:heatSL} and \eqref{eq:ReactionDiffusionSL} are locally well-posed, see \Cref{Th:WPLinftySL} and \Cref{Th:WPLinftySLReactionDiffusion} below. \Cref{sec:NullControlHeatSlSmooth} and \Cref{sec:ProofMainresults} are devoted to the proofs of the main results, i.e. \Cref{th:Mainresult1} and \Cref{th:mainresult2}.

\section{Well-posedness and regularity results for the heat equation}
\label{sec:WPparabolicresults}
In this section, we give the notion of solutions that we consider in what follows. 
Then we recall standard well-posedness results for both linear and semi-linear heat equations in various functional settings we will use in what follows.
\subsection{Functional spaces}
In this article, we use in a crucial way a $L^p$ framework with $p\in (1,\infty)$. First, we introduce the standard notation 
for the dual exponent $p' \in (1, \infty)$ of $p$ defined by the relation 
$$\frac{1}{p}+\frac{1}{p'}=1.$$
We also introduce the following functional spaces
\begin{equation}\label{eq:defXp}
\mathcal{X}^p := L^p(0,T;W^{2,p}(\Omega)) \cap W^{1,p}(0,T;L^p(\Omega)).
\end{equation}
We have the following classical embedding results (see, for instance, \cite[Lemma 3.3, p.80]{LSU68}): for $p, q\geq 1$, 
\begin{equation}
\label{XpLq}
\mathcal{X}^p \hookrightarrow L^{q}(0,T;L^{q}(\Omega)) \quad \text{if} \quad \frac{1}{q}\geq \frac{1}{p}-\frac{2}{d+2},
\quad
\mathcal{X}^p \hookrightarrow L^{\infty}(0,T;L^{\infty}(\Omega)) \quad \text{if} \quad p > \frac{d+2}{2},
\end{equation}
\begin{equation}
\label{XpWq}
\mathcal{X}^p \hookrightarrow L^{q}(0,T;W^{1,q}(\Omega)) \quad \text{if} \quad \frac{1}{q}\geq \frac{1}{p}-\frac{1}{d+2},
\quad
\mathcal{X}^p \hookrightarrow L^{\infty}(0,T;W^{1,\infty}(\Omega)) \quad \text{if} \quad p > d+2.
\end{equation}
We also have, see for instance \cite[Lemma 3.4, p.82]{LSU68},
\begin{equation}
\label{eq:embeddingXpContinuousFrac}
\mathcal{X}^p \hookrightarrow C^0([0,T];W^{2/p',p}(\Omega)),
\end{equation}
where $W^{\alpha,p}(\Omega)$ denotes the fractional Sobolev spaces (see, for instance, \cite[p.70]{LSU68}).
We recall that functions in $W^{\alpha,p}(\Omega)$ admit a trace on $\partial\Omega$ if $\alpha>1/p$. 
If $\alpha>1/p$, we denote by $W_0^{\alpha,p}(\Omega)$ the subspace of functions $f\in W^{\alpha,p}(\Omega)$ such that $f=0$ on $\partial\Omega$.
We also write $W_0^{\alpha,p}(\Omega):=W^{\alpha,p}(\Omega)$ if $\alpha\leq 1/p$.
From \cite[Corollary 4.53, p.216]{Demengel}, we have 
$$
W^{2/p',p}(\Omega) \hookrightarrow L^\infty(\Omega)\quad \text{if} \quad p > \frac{d+2}{2},
$$
and thus
\begin{equation}
\label{10:06}
\mathcal{X}^p \hookrightarrow C^0([0,T];L^\infty(\Omega))\quad \text{if} \quad p > \frac{d+2}{2}.
\end{equation}

We finish with some other classical results on the spaces $\mathcal{X}^p$, for which we give a short proof for completeness.
\begin{Lemma}\label{L01}
The following statements hold.
\begin{enumerate}
\item If $p > \frac{d+2}{2}$, then $\mathcal{X}^p$ is an algebra.
\item For any $N>1$, $q>1$, if
\begin{equation}\label{11:44}
\frac{1}{q}\left(1-\frac{1}{N}\right)< \frac{2}{2+d}
\end{equation}
then the embedding 
\begin{equation}\label{eq:embeddingXqLnq}
\mathcal{X}^q {\hookrightarrow} L^{N q}((0,T)\times \Omega) \quad \text{is compact}.
\end{equation}
\end{enumerate}
\end{Lemma}
\begin{proof}
For the first point, we consider $f,g\in \mathcal{X}^p$. Then
$$
\partial_t f, \ \partial_t g, \ \nabla^2 f, \ \nabla^2 g\in L^p(0,T;L^{p}(\Omega)),
$$
and from \eqref{XpLq} and \eqref{XpWq}
$$
f,\ g\in L^{\infty}(0,T;L^{\infty}(\Omega)), \quad \nabla f,\  \nabla g \in L^{2p}(0,T;L^{2p}(\Omega)).
$$
We thus deduce that 
$$
\partial_t(fg), \  \nabla^2 (fg) \in L^p(0,T;L^{p}(\Omega)).
$$
For the second point, we can use \eqref{11:44} to consider $p>1$ such that 
\begin{equation}\label{11:44-2}
\frac{1}{qN}>\frac{1}{p} > \frac{1}{q}- \frac{2}{2+d}.
\end{equation}
We thus deduce from \eqref{XpLq} that
$$
\mathcal{X}^q \hookrightarrow L^{p}((0,T)\times \Omega) \hookrightarrow L^{N q}((0,T)\times \Omega)
$$
and from the H\"older inequality, there exists $\theta\in (0,1)$ such that 
\begin{equation}\label{20:50}
\left\| f \right\|_{L^{N q}((0,T)\times \Omega)} \leq \left\| f \right\|_{L^{q}((0,T)\times \Omega)}^{\theta} \left\| f \right\|_{L^{p}((0,T)\times \Omega)}^{1-\theta}
\quad (f\in L^{p}((0,T)\times \Omega)).
\end{equation}
From the Aubin-Lions lemma (see, for instance, \cite[Section 8, Corollary 4]{Sim87}), 
the embedding 
$$
\mathcal{X}^q \hookrightarrow L^{q}((0,T)\times \Omega) \quad \text{is compact}.
$$
Consequently, if $\left( f_n\right)$ is a bounded sequence of $\mathcal{X}^q$, it has a subsequence converging in $L^{q}((0,T)\times \Omega)$ and bounded in 
$L^{p}((0,T)\times \Omega)$. From \eqref{20:50}, this subsequence is converging in $L^{qN}((0,T)\times \Omega)$.
\end{proof}

\subsection{Linear heat equation}
Let us first consider the linear nonhomogenenous heat equation
\begin{equation}
\label{eq:heatL}
\left\{
\begin{array}{ll}
\partial_t y - \Delta y = g &\mathrm{in}\ (0,T)\times\Omega,\\
y= 0&\mathrm{on}\ (0,T)\times\partial\Omega,\\
y(0,\cdot)=y_{0}& \mathrm{in}\ \Omega.
\end{array}
\right.
\end{equation}
In this article, we need several definitions of solutions for \eqref{eq:heatL}:
\begin{Definition}\label{D3}
We introduce three concepts of solutions for \eqref{eq:heatL}.
\begin{enumerate}
\item If $y_0\in W_0^{2/p',p}(\Omega)$ and $g \in L^p((0,T)\times \Omega)$, we say that $y\in \mathcal{X}^p$ is a strong solution of \eqref{eq:heatL}
if it satisfies \eqref{eq:heatL} a.e. and in the trace sense.
\item If $y_0 \in L^2(\Omega)$ and $g \in L^2(0,T;H^{-1}(\Omega))$, we say that $y \in L^2(0,T;H_0^1(\Omega)) \cap H^1(0,T;H^{-1}(\Omega))$ is a weak solution
if 
\begin{multline}
\label{eq:formvarL}
\int_0^T \langle \partial_t y (t,\cdot), \zeta(t,\cdot) \rangle)_{H^{-1}(\Omega), H_0^1(\Omega)} dt 
+ \int_{0}^T \int_{\Omega} \nabla y(t,x) \cdot \nabla \zeta(t,x) dt dx
\\ 
= \int_0^T \langle g(t,\cdot), \zeta(t,\cdot) \rangle_{H^{-1}(\Omega), H_0^1(\Omega)} dt\qquad \forall \zeta \in L^2(0,T;H_0^1(\Omega)),
\end{multline}
and
\begin{equation}
\label{eq:dataL2}
y(0,\cdot) = y_0 \ \text{in}\ L^2(\Omega).
\end{equation}
\item If $y_0 \in L^1(\Omega)$ and $g \in L^{1}((0,T)\times \Omega)$, we say that $y \in L^p((0,T)\times \Omega)$ is a very weak solution of \eqref{eq:heatL} if
$$
\iint_{(0,T)\times \Omega} y (-\partial_t \zeta - \Delta \zeta) dt dx 
=\iint_{(0,T)\times \Omega} g  \zeta  dt dx 
+ \int_{\Omega} y_0(x) \zeta(0,x) dx \qquad \forall \zeta\in C^\infty_c([0,T)\times \Omega).
$$
\end{enumerate}
\end{Definition}
We recall the following implications 
$$ \text{strong solution}\ \implies \text{weak solution}\ \implies\ \text{very weak solution},$$
and the reverse implications are also true assuming that $y$ is regular enough. We also note that the definition of weak solution is meaningful due to the continuous embedding
(see, for instance, \cite[Theorem 3, p.303]{Eva10})
\begin{equation}
\label{eq:embbedingL2}
L^2(0,T;H_0^1(\Omega)) \cap H^1(0,T;H^{-1}(\Omega)) \hookrightarrow C^0([0,T];L^2(\Omega)).
\end{equation}
We also state standard results for the well-posedness of \eqref{eq:heatL} (see, for instance \cite[Theorems 3 and 4, pp.378-379]{Eva10},
\cite[Theorem 7.1, p.181]{LSU68} and \cite[Theorem 9.1, p.341]{LSU68}):
\begin{Theorem}\label{T03}
The following well-posedness results hold.
\begin{enumerate}
\item For any $y_0 \in L^2(\Omega)$ and $g \in L^2(0,T;H^{-1}(\Omega))$, the equation \eqref{eq:heatL} admits a unique weak solution 
$y$ and we have the estimate 
\begin{equation}
\label{eq:estimateL2}
\norme{y}_{L^2(0,T;H_0^1(\Omega)) \cap H^1(0,T;H^{-1}(\Omega))} \lesssim \norme{y_0}_{L^2(\Omega)} + \norme{g}_{L^2(0,T;H^{-1}(\Omega))}.
\end{equation}
\item For $y_0 \in L^{\infty}(\Omega)$ and $g \in L^{\infty}((0,T)\times \Omega)$, the unique weak solution $y$ of \eqref{eq:heatL} satisfies
\begin{equation}
\label{eq:estimationLinftyheatL}
\norme{y}_{L^{\infty}((0,T)\times \Omega)} \lesssim \norme{y_0}_{L^{\infty}(\Omega)} + \norme{g}_{L^{\infty}((0,T)\times \Omega)}.
\end{equation}
\item Assume $p\in (1,\infty)$. For any $y_0\in W_0^{2/p',p}(\Omega)$ and $g \in L^p((0,T)\times \Omega)$, there
exists a unique strong solution $y\in \mathcal{X}^p$ of \eqref{eq:heatL} and we have the estimate 
\begin{equation}
\label{eq:EstimationLpheatL}
\left\| y \right\|_{\mathcal{X}^p} \lesssim \left\| y_0 \right\|_{W^{2/p',p}(\Omega)} + \left\| g \right\|_{L^p((0,T)\times \Omega)}.
\end{equation}
\end{enumerate}
\end{Theorem}

\subsection{Semi-linear heat equation}

For $N\in \mathbb{N}$, $N \geq 2$, let us then consider the semi-linear heat equation
\begin{equation}
\label{eq:heatSLg}
\left\{
\begin{array}{ll}
\partial_t y - \Delta y = y^N + g &\mathrm{in}\ (0,T)\times\Omega,\\
y= 0&\mathrm{on}\ (0,T)\times\partial\Omega,\\
y(0,\cdot)=y_{0}& \mathrm{in}\ \Omega.
\end{array}
\right.
\end{equation}
The space $\mathcal{W}$ is defined in \eqref{21:47}.
First we recall the definition of a weak solution for the system \eqref{eq:heatSLg}:
\begin{Definition}\label{D1}
We say that $y\in \mathcal{W}$
is a weak solution of \eqref{eq:heatSLg} if
\begin{multline}
\label{eq:formvarSL}
\int_0^T \langle \partial_t y (t,\cdot), \zeta(t,\cdot) \rangle)_{H^{-1}(\Omega), H_0^1(\Omega)} dt 
+ \int_{0}^T \int_{\Omega} \nabla y(t,x) \cdot \nabla \zeta(t,x) dt dx
\\ 
= \int_0^T \int_{\Omega} y^N(t,x) \zeta(t,x) dt dx 
+  \int_0^T \langle g(t,\cdot), \zeta(t,\cdot) \rangle_{H^{-1}(\Omega), H_0^1(\Omega)} dt\qquad \forall \zeta \in L^2(0,T;H_0^1(\Omega)),
\end{multline}
and
\begin{equation}
\label{eq:dataL2SL}
y(0,\cdot) = y_0 \ \text{in}\ L^2(\Omega).
\end{equation}
\end{Definition}

Let us state the following well-posedness result for \eqref{eq:heatSLg} for small data. This result is standard, but we recall the proof for completeness.
\begin{Theorem}
\label{Th:WPLinftySL}
Assume $p$ satisfies \eqref{09:42}.
There exists $\delta >0$ small enough such that for any $y_0 \in L^{\infty}(\Omega)$ and $g \in L^p((0,T)\times \Omega)$, satisfying
\begin{equation}
\label{eq:smallinitialsource}
\norme{y_0}_{L^{\infty}(\Omega)} + \norme{g}_{L^p((0,T)\times \Omega)} \leq \delta,
\end{equation}
the system \eqref{eq:heatSLg} admits a unique weak solution. Moreover, we have
\begin{equation}
\label{eq:smallnonlinearityLinfty}
\norme{y}_{\mathcal{W}}+\norme{y^N}_{L^{\infty}((0,T)\times \Omega)} \lesssim \norme{y_0}_{L^{\infty}(\Omega)} + \norme{g}_{L^p((0,T)\times \Omega)}.
\end{equation}
\end{Theorem}
\begin{proof}
First, we show that for any $F\in L^{\infty}((0,T)\times \Omega)$, there exists a unique weak solution
to the heat equation
\begin{equation}
\label{eq:heatL-A}
\left\{
\begin{array}{ll}
\partial_t y - \Delta y = F+g &\mathrm{in}\ (0,T)\times\Omega,\\
y= 0&\mathrm{on}\ (0,T)\times\partial\Omega,\\
y(0,\cdot)=y_{0}& \mathrm{in}\ \Omega.
\end{array}
\right.
\end{equation}
In order to do this, we can write $y=y_1+y_2$ with
\begin{equation}
\label{eq:heatL-B}
\left\{
\begin{array}{ll}
\partial_t y_1 - \Delta y_1 = F &\mathrm{in}\ (0,T)\times\Omega,\\
y_1= 0&\mathrm{on}\ (0,T)\times\partial\Omega,\\
y_1(0,\cdot)=y_{0}& \mathrm{in}\ \Omega,
\end{array}
\right.
\quad
\left\{
\begin{array}{ll}
\partial_t y_2 - \Delta y_2 = g &\mathrm{in}\ (0,T)\times\Omega,\\
y_2= 0&\mathrm{on}\ (0,T)\times\partial\Omega,\\
y_2(0,\cdot)=0& \mathrm{in}\ \Omega,
\end{array}
\right.
\end{equation}
Applying \Cref{T03}, the above systems admit respectively a unique solution $y_1\in \mathcal{W}$ and $y_2\in \mathcal{X}^p$
and with the hypotheses on $p$, we deduce from \eqref{XpLq} that $\mathcal{X}^p\hookrightarrow \mathcal{W}$.
We conclude the existence and the uniqueness of a weak solution $y\in \mathcal{W}$ of \eqref{eq:heatL-A} and we have the estimate 
\begin{equation}\label{19:30}
\left\| y \right\|_{L^{\infty}((0,T)\times \Omega)} \lesssim \norme{y_0}_{L^{\infty}(\Omega)} + \norme{g}_{L^p((0,T)\times \Omega)} +
\left\| F \right\|_{L^{\infty}((0,T)\times \Omega)}.
\end{equation}
We can thus define the following mapping 
\begin{equation}
\label{eq:defNheatSlg}
\mathcal{N} : L^{\infty}((0,T)\times \Omega) \to L^{\infty}((0,T)\times \Omega), \quad F\mapsto y^N,
\end{equation}
where $y$ is the unique weak solution to \eqref{eq:heatL-A} and if $y_0$ and $g$ satisfy \eqref{eq:smallinitialsource}
and if we consider 
\begin{equation}
\label{eq:smallballLinfty}
B_\delta  := \{F \in L^{\infty}((0,T)\times \Omega)\ ;\ \norme{F}_{L^{\infty}((0,T)\times \Omega)} \leq \delta\},
\end{equation}
then we deduce from \eqref{19:30} that for $\delta>0$ small enough, $\mathcal{N}(B_\delta)\subset B_\delta$. We can also show a similar way that
the restriction of $\mathcal{N}$ on $B_\delta$ is a strict contraction. The Banach fixed point yields the existence of a unique fixed point $F$ and the corresponding solution
$y$ of \eqref{eq:heatL-A} is a weak solution of \eqref{eq:heatSLg}.

For the uniqueness, we consider $y_1,y_2\in \mathcal{W}$ two solutions of \eqref{eq:heatSLg}. Then, $y:=y_1- y_1$ satisfies (in a weak sense)
\begin{equation}
\label{eq:heatSLgDiff}
\left\{
\begin{array}{ll}
\partial_t {y} - \Delta {y} = y_1^N - {y}_2^N &\mathrm{in}\ (0,T)\times\Omega,\\
{y}= 0&\mathrm{on}\ (0,T)\times\partial\Omega,\\
{y}(0,\cdot)=0& \mathrm{in}\ \Omega.
\end{array}
\right.
\end{equation}
In particular, using that $y_1,y_2\in L^\infty((0,T)\times \Omega)$, we can write the standard energy estimate: for any $t\in [0,T]$,
$$
\left\| y(t,\cdot) \right\|_{L^2(\Omega)}^2 \leq \int_0^t\int_\Omega y \left( y_1^N - {y}_2^N \right)  ds dx
\lesssim \int_0^t \left\| y(s,\cdot) \right\|_{L^2(\Omega)}^2 ds,
$$
and we conclude with the Gr\"onwall lemma.
\end{proof}
We now state some regularizing effects of \eqref{eq:heatSLg}.
\begin{Lemma}\label{L20}
Assume the same hypotheses of \cref{Th:WPLinftySL} and let us consider $y$ the corresponding weak solution of \eqref{eq:heatSLg}.
\begin{enumerate}
\item If $g=0$ then for any $t\in (0,T]$ and for any $q>1$, $y(t,\cdot)\in W^{2/q',q}_0(\Omega)$. Moreover, we have the estimate
\begin{equation}
\label{eq:smallyTWq'}
\norme{y(t,\cdot)}_{W^{2/q',q}(\Omega)} \lesssim_{t,q} \norme{y_0}_{L^{\infty}(\Omega)}.
\end{equation}
\item In the general case, for any $t\in (0,T]$,  $y(t,\cdot)\in L^\infty(\Omega)$ and we have the estimate
\begin{equation}
\label{eq:smallyTWq'2}
\norme{y(t,\cdot)}_{L^\infty(\Omega)} \lesssim_{t,q} \norme{y_0}_{L^{\infty}(\Omega)}+ \norme{g}_{L^p((0,T)\times \Omega)}.
\end{equation}
\end{enumerate}
\end{Lemma}
\begin{proof}
Let us denote by $\theta$ the function $\theta(t)=t$ for $t\in \mathbb{R}$. Then we deduce from \eqref{eq:heatSLg} that
$$
\left\{
\begin{array}{ll}
\partial_t (\theta y) - \Delta (\theta y) = y +\theta y^N &\mathrm{in}\ (0,T)\times\Omega,\\
\theta y= 0&\mathrm{on}\ (0,T)\times\partial\Omega,\\
(\theta y)(0,\cdot)=0& \mathrm{in}\ \Omega
\end{array}
\right.
$$
and from \cref{Th:WPLinftySL}, 
$$
\left\| y +\theta y^N  \right\|_{L^\infty((0,T)\times \Omega)} \lesssim \left\| y_0 \right\|_{L^\infty(\Omega)}.
$$
Applying \cref{T03}, we deduce that $\theta y\in \mathcal{X}^q$ for any $q>1$, and we conclude with \eqref{eq:embeddingXpContinuousFrac}.

The second point can be done similarly by using \eqref{10:06} and that $p$ satisfies \eqref{09:42}.
\end{proof}

The above definition and properties can be extended to the parabolic system
\begin{equation}
\label{eq:ReactionDiffusionSLSourceg}
\left\{
\begin{array}{ll}
\partial_t y_1 - \Delta y_1 =y_1^{N_1} + g &\mathrm{in}\ (0,T)\times\Omega,\\
\partial_t y_2 - \Delta y_2 = y_1^{N_2}+ y_2^{N_3}  &\mathrm{in}\ (0,T)\times\Omega,\\
y_1=y_2= 0&\mathrm{on}\ (0,T)\times\partial\Omega,\\
y_1(0,\cdot)=y_{1,0}, \quad y_2(0,\cdot)=y_{2,0}& \mathrm{in}\ \Omega.
\end{array}
\right.
\end{equation}
More precisely, we have the following definition and well-posedness results:
\begin{Definition}\label{D2}
We say that $(y_1,y_2)\in \mathcal{W}\times \mathcal{W}$ is a weak solution of \eqref{eq:ReactionDiffusionSLSourceg} if
\begin{multline*}
\int_0^T \langle \partial_t y_1 (t,\cdot), \zeta(t,\cdot) \rangle)_{H^{-1}(\Omega), H_0^1(\Omega)} dt 
+ \int_{0}^T \int_{\Omega} \nabla y_1(t,x) \cdot \nabla \zeta(t,x) dt dx
\\ 
= \int_0^T \int_{\Omega} y_1^{N_1}(t,x) \zeta(t,x) dt dx 
+  \int_0^T \langle g(t,\cdot), \zeta(t,\cdot) \rangle_{H^{-1}(\Omega), H_0^1(\Omega)} dt\qquad \forall \zeta \in L^2(0,T;H_0^1(\Omega)),
\end{multline*}
\begin{multline*}
\int_0^T \langle \partial_t y_2 (t,\cdot), \zeta(t,\cdot) \rangle)_{H^{-1}(\Omega), H_0^1(\Omega)} dt 
+ \int_{0}^T \int_{\Omega} \nabla y_2(t,x) \cdot \nabla \zeta(t,x) dt dx
\\ 
= \int_0^T \int_{\Omega} y_1^{N_2}(t,x) \zeta(t,x) dt dx +\int_0^T \int_{\Omega} y_2^{N_3}(t,x) \zeta(t,x) dt dx 
\qquad \forall \zeta \in L^2(0,T;H_0^1(\Omega)),
\end{multline*}
and
\begin{equation}
\label{eq:dataL2SLSystem}
(y_1,y_2)(0,\cdot) = (y_{1,0},y_{2,0}) \ \text{in}\ L^2(\Omega)^2.
\end{equation}
\end{Definition}

\begin{Theorem}
\label{Th:WPLinftySLReactionDiffusion}
Assume $p$ satisfies \eqref{09:42}. 
There exists $\delta >0$ small enough such that for any $(y_{1,0},y_{2,0}) \in L^{\infty}(\Omega)^2$ and 
$g \in L^p((0,T)\times \Omega)$, satisfying
\begin{equation}
\label{eq:smallinitialsourceSystem}
\norme{(y_{1,0},y_{2,0})}_{L^{\infty}(\Omega)^2} + \norme{g}_{L^p((0,T)\times \Omega)} \leq \delta,
\end{equation}
the system \eqref{eq:ReactionDiffusionSLSourceg} admits a unique weak solution. Moreover, we have
\begin{multline}
\label{eq:smallnonlinearityLinftySystem}
\left\| y_1 \right\|_{\mathcal{W}}+\left\| y_2 \right\|_{\mathcal{W}}+
\norme{y_1^{N_1}}_{L^{\infty}((0,T)\times \Omega)} + \norme{y_1^{N_2}}_{L^{\infty}((0,T)\times \Omega)} + \norme{y_2^{N_3}}_{L^{\infty}((0,T)\times \Omega)} 
\\
\lesssim \norme{y_0}_{L^{\infty}(\Omega)} + \norme{g}_{L^p((0,T)\times \Omega)}.
\end{multline}
\end{Theorem}


\section{Proof of \cref{th:mainresult2}}
\label{sec:NullControlHeatSlSmooth}

The goal of this part is to prove \cref{th:mainresult2}. 

We first set
\begin{equation}
\label{eq:rho0}
\rho_0(t):= \exp\left(- \frac{1}{t(T-t)}\right)\quad (t\in (0,T)), 
\quad \rho(0)=\rho(T)=0,
\end{equation}
and
\begin{equation}
\label{eq:rho}
\rho(t):= 
\left\{
\begin{array}{l}
\exp\left(- \frac{1}{(T/2)^2}\right)\quad (t\in [0,T/2)), \\
\exp\left(- \frac{1}{t(T-t)}\right)\quad (t\in [T/2,T)), 
\end{array}
\right.
\quad \rho(T)=0.
\end{equation}

Using \cref{L20}, that is taking the control $h \equiv 0$ in $[0,T/2]\times\omega$ in order to benefit from the regularizing effect of the semi-linear heat equation \eqref{eq:heatSL}, we see that it is sufficient to show the following result.
\begin{Theorem}\label{th:MainresultSmooth}
Assume $N, n\in \mathbb{N}$, $N\geq 2$ and $T>0$. 
Let us consider $p$ satisfying
\begin{equation}\label{eq:defpOdd}
p=(2n+1)(2k+1)+1,
\end{equation}
with $k\in \mathbb{N}$ large enough so that
\begin{equation}\label{eq:defpgrand}
p>\frac{d+2}{2},
\end{equation}
and $q\geq p'$ satisfying \eqref{11:44}.
There exist $\delta>0$ and $m>0$ such that for any initial data $y_0\in W^{\frac{2}{q'},q}_0(\Omega)$ with
$$
\left\| y_0 \right\|_{W^{\frac{2}{q'},q}(\Omega)} \leq \delta,
$$ 
there exists a control $h$ and a strong solution $y$ of \eqref{eq:heatSL} such that
\begin{equation}\label{10:56}
\frac{y}{\rho^{m}} \in \mathcal X^q, \quad y^N\in L^q(0,T;L^q(\Omega)), \quad
h \in L^{\infty}(0,T;L^\infty(\Omega)), \quad \left(\frac{h}{\rho_0^{m} }\right)^{1/(2n+1)} \in \mathcal X^p,
\quad
\end{equation}
together with the estimate
\begin{equation}\label{16:59-Result-2}
\norme{ \frac{y}{\rho^{m}} }_{\mathcal{X}^{q}} +\left\| \frac{y^N}{\rho^{m_1}} \right\|_{L^q(0,T;L^q(\Omega))}^{1/N}
+\left\| \frac{h}{\rho_0^{m}} \right\|_{L^{\infty}(0,T;L^\infty(\Omega))}
+\norme{\left(\frac{h}{\rho_0^{m} }\right)^{1/(2n+1)} }_{\mathcal X^p}^{2n+1}
\lesssim 
 \norme{y_0}_{W^{2/q',q}(\Omega)}.
\end{equation}
In particular, $h$ satisfies \eqref{eq:smallcontrolheatSL} and \eqref{eq:oddcontrol} and $y$ satisfies \eqref{eq:nully}. 
\end{Theorem}

\subsection{Carleman estimate and $L^2$ observability inequality for the heat equation}
The goal of this part is to deduce a weighted $L^2$ observability inequality for the heat equation from a Carleman estimate.
We first recall a standard Carleman estimate for the heat equation that is due to Fursikov and Imanuvilov \cite{FI}. 
We start by introducing a nonempty domain $\omega_0$ such that $\chi_\omega>0$ on $\overline{\omega_0}\subset\omega$. 
By using \cite{FI}, see also \cite[Theorem 9.4.3]{TucsnakWeiss}, there exists
$\eta^0\in C^2(\overline{\Omega})$ satisfying
\begin{equation}
\label{eq:defeta0}
\eta^0>0 \ \text{in}\ \Omega, \quad 
\eta^0= 0 \ \text{on} \ \partial\Omega,\quad
\max_\Omega \eta^0=1,\quad
\nabla\eta^0\neq 0 \ \text{in}\ \overline{\Omega\setminus \omega_0}.
\end{equation}
We then define the following functions:
\begin{equation} \label{eq:weights}
\alpha(t,x)=\frac{\exp\left(4\lambda\right)-\exp\{\lambda (2+\eta^0(x))\}}{t(T-t)},
	\quad
\xi(t,x)=\frac{\exp\{\lambda (2+\eta^0(x))\}}{t(T-t)}.
\end{equation}
We can now state the Carleman estimate for the heat equation, see \cite[Lemma 1.3]{FG}.
\begin{Theorem}\label{T04}
There exist $\lambda_0, s_0, C_0\in \mathbb{R}_+^*$ such that for any $\lambda\geq \lambda_0$, $s\geq s_0(T+T^2)$, 
$\zeta\in \mathcal{X}^2$ with $\zeta=0$ on $(0,T)\times \partial \Omega$,
\begin{multline}\label{eq:dCarlemanheatL2}
\iint_{(0,T)\times \Omega} s^3\lambda^4 \xi^3 e^{-2s\alpha} \left| \zeta \right|^2 \ dx dt 
\\
\leq C_0 \left( 
\iint_{(0,T)\times \Omega} e^{-2s\alpha} \left| \partial_t\zeta+\Delta \zeta \right|^2 \ dx dt
+\iint_{(0,T)\times \Omega} s^3\lambda^4 \xi^3 e^{-2s\alpha} \left| \chi_{\omega} \zeta \right|^2 \ dx dt
\right).
\end{multline}
\end{Theorem}
From the above result, one can obtain a similar estimate with weights depending only on time. We recall that $\rho_0$ and $\rho$ are defined in \eqref{eq:rho0} and \eqref{eq:rho}. We have that $\rho_0,\rho\in W^{1,\infty}(0,T)\cap C^0([0,T])$ and 
\begin{equation}
\label{eq:comparaisonrho0rho}
\rho_0 \leq \rho\leq 1,\qquad \left|\left(\frac{\rho_0}{\rho}\right)'\right| \lesssim 1.
\end{equation}
Moreover, we have the following instrumental estimates
\begin{equation}
\label{eq:comparaisonrhom1m2}
m_1 < m_2 \Rightarrow \left(\rho^{m_2} \leq \rho^{m_1},\ |(\rho^{m_2})'| \lesssim \rho^{m_1}\right).
\end{equation}

With the above notation, we can state the following corollary of \cref{T04}.
\begin{Corollary}\label{C1}
Assume $r>1$. Then, there exist $m_0, M_0 \in \mathbb{R}_+^*$ with
\begin{equation}\label{eq:m0M0r0}
m_0 < M_0 < r m_0,
\end{equation}
such that for any $\zeta\in \mathcal{X}^2$ with $\zeta=0$ on $(0,T)\times\partial\Omega$ the following relation holds
\begin{equation}
\label{eq:ObsL2}
\norme{\zeta(0,\cdot)}_{L^2(\Omega)}+ \norme{\rho^{M_0} \zeta}_{L^2(0,T;L^2(\Omega))}
\lesssim
\norme{\rho^{m_0}\pare{\partial_t \zeta+\Delta \zeta} }_{L^2(0,T;L^2(\Omega))}
+ \norme{\rho_0^{m_0} \zeta \chi_{\omega}}_{L^2(0,T;L^2(\Omega))}.
\end{equation}
\end{Corollary}
We want to highlight the fact that the dependence in space of the Carleman weights appearing in \eqref{eq:dCarlemanheatL2} has been removed in \eqref{eq:ObsL2}. Moreover, it is worth mentioning that the vanishing property at $t=T$ of the Carleman weights for the left-hand-side of \eqref{eq:dCarlemanheatL2} has been dropped. This is why one can make appeared the first left-hand-side of \eqref{eq:ObsL2}, that is the classical left-hand-side term for proving a $L^2$ observability inequality. The same remark applies for the first right-hand-side term of \eqref{eq:dCarlemanheatL2} to get the first right-hand-side term of \eqref{eq:ObsL2}. Finally, the fact that $m_0$ and $M_0$ are comparable is quantified in \eqref{eq:m0M0r0}.
\begin{proof}[Proof of \cref{C1}]
We consider $s_0$ and $\lambda_0$ from \Cref{T04}. Then, we deduce from \eqref{eq:defeta0} and \eqref{eq:weights} that
for any $\lambda\geq \lambda_0$, $s\geq s_0(T+T^2)$, 
$$
1 \lesssim s^3\lambda^4 \xi^3 \lesssim (s \lambda^2 \xi)^3 \lesssim e^{s\lambda^2\frac{e^{3\lambda}}{t(T-t)}}.
$$
Therefore combining these estimates with \eqref{eq:defeta0} and \eqref{eq:weights} and taking $s=s_0(T+T^2)$,
we deduce that
$$
\rho_0^{M_0} \lesssim s^3\lambda^4 \xi^3 e^{-s\alpha},
\quad 
e^{-s\alpha}\lesssim s^3\lambda^4 \xi^3 e^{-s\alpha}\lesssim  \rho_0^{m_0}
$$
with
$$
M_0:=s_0(T+T^2)\left(e^{4\lambda}-e^{2\lambda}\right), \quad m_0:=s_0(T+T^2)\left(e^{4\lambda}-e^{3\lambda} (1+\lambda^2) \right).
$$
We now fix $\lambda=\lambda_0$ large enough, so that \eqref{eq:m0M0r0} holds. Applying \eqref{eq:dCarlemanheatL2}, we obtain
\begin{equation}\label{13:13}
\norme{\rho_0^{M_0} \zeta}_{L^2(0,T;L^2(\Omega))}
\lesssim
\norme{\rho_0^{m_0}\pare{\partial_t \zeta+\Delta \zeta} }_{L^2(0,T;L^2(\Omega))}
+ \norme{\rho_0^{m_0} \zeta \chi_{\omega}}_{L^2(0,T;L^2(\Omega))}.
\end{equation}
Using \eqref{eq:rho0} and \eqref{eq:rho}, the above relation yields
\begin{equation}\label{13:13-a}
\norme{\rho^{M_0} \zeta}_{L^2(T/2,T;L^2(\Omega))}
\lesssim
\norme{\rho_0^{m_0}\pare{\partial_t \zeta+\Delta \zeta} }_{L^2(0,T;L^2(\Omega))}
+ \norme{\rho_0^{m_0} \zeta \chi_{\omega}}_{L^2(0,T;L^2(\Omega))}.
\end{equation}
Let us consider $\chi_T\in C^\infty([0,T])$, $\chi_T \equiv 1$ in $[0,T/2]$, $\chi_T\equiv 0$ in $[3T/4,T]$ and $\left| \chi_T' \right| \lesssim 1/T$.
Then
\begin{equation}
\label{eq:chiT}
\left\{
\begin{array}{rl}
- \partial_t \pare{\chi_T\zeta} -  \Delta \pare{\chi_T\zeta} = -\left(\chi_T\right)'\zeta - \chi_T \left(\partial_t \zeta+\Delta \zeta \right) &\mathrm{in}\ (0,T)\times\Omega,\\
\pare{\chi_T\zeta}= 0&\text{on}\ (0,T)\times\partial\Omega,\\
\pare{\chi_T\zeta}(T,\cdot)=0& \text{in}\ \Omega.
\end{array}
\right.
\end{equation}
By using the maximal regularity of the heat equation in $L^2$ i.e. \Cref{T03} with $p=2$ to \eqref{eq:chiT} and the Sobolev embedding 
\eqref{eq:embeddingXpContinuousFrac} we deduce
$$
\norme{\zeta(0,\cdot)}_{L^2(\Omega)}+
\norme{\zeta}_{L^2(0,T/2;L^2(\Omega))}
\lesssim
 \norme{{\partial_t \zeta+\Delta \zeta} }_{L^2(0,3T/4;L^2(\Omega))}
+\norme{\zeta}_{L^2(T/2,3T/4;L^2(\Omega))},
$$
and thus by using that $\rho(t)=\rho(T/2)$ in $(0,T/2)$ and $\rho(t)\geq \rho(3T/4)$ in $(0,3T/4)$, we obtain 
$$
\norme{\zeta(0,\cdot)}_{L^2(\Omega)}+
\norme{\rho^{M_0} \zeta}_{L^2(0,T/2;L^2(\Omega))}
\lesssim
 \norme{\rho^{m_0} ({\partial_t \zeta+\Delta \zeta} )}_{L^2(0,T;L^2(\Omega))}
+\norme{\rho^{M_0} \zeta}_{L^2(T/2,T;L^2(\Omega))}.
$$
Combining this last estimate with \eqref{13:13-a} and \eqref{eq:comparaisonrho0rho}, we deduce the expected observability inequality \eqref{eq:ObsL2}.
\end{proof}

\subsection{A weighted $L^p$ observability inequality}

The goal of this part is to deduce from the weighted $L^2$ observability inequality in \Cref{C1} a weighted $L^p$ observability inequality for $p \geq 2$, by applying maximal regularity results for the heat equation. More precisely, we show the following result:
\begin{Proposition}\label{L4Obs}
Assume $p\geq 2$ and $r\in (1,p')$. Then, there exist $m_0, m_1 \in \mathbb{R}_+^*$ with
\begin{equation}\label{10:09}
m_0 < m_1 < r m_0,
\end{equation}
such that for any $\zeta\in \mathcal{X}^p$ with $\zeta=0$ on $(0,T)\times\partial\Omega$, the following relation holds
\begin{equation}
\label{eq:ObsLp}
\norme{\zeta(0,\cdot)}_{L^p(\Omega)}
+ \norme{\rho^{m_1} \zeta}_{L^p(0,T;L^p(\Omega))}
\lesssim
\norme{\rho^{m_0}\pare{\partial_t \zeta+\Delta \zeta} }_{L^p(0,T;L^p(\Omega))}
+ \norme{\rho_0^{m_0} \zeta \chi_{\omega}}_{L^p(0,T;L^p(\Omega))}.
\end{equation}
\end{Proposition}
The main difference between \eqref{eq:ObsLp} and \eqref{eq:ObsL2} is the $L^p$ framework. We want to highlight that $M_0$ of \eqref{eq:m0M0r0} has been transformed into $m_1$ of \eqref{10:09}. Basically, the proof is as follows. By a bootstrap argument, we apply recursively maximal regularity results in $L^r$, starting from $r=2$ together with Sobolev embeddings to obtain \eqref{eq:ObsLp}. During the induction process, $M_0$ becomes $M_1 \in (M_0, r m_0)$ then $M_2 \in (M_1,r m_0)$, etc. to finally take the value $m_1 \in (m_0,rm_0)$.
\begin{proof}
First, we apply \cref{C1} to obtain $m_0, M_0 \in \mathbb{R}_+^*$ satisfying \eqref{eq:m0M0r0}
and such that \eqref{eq:ObsL2} holds for any $\zeta\in \mathcal{X}^2$ with $\zeta=0$ on $(0,T)\times\partial\Omega$.
We then set $g:=- \partial_t \zeta-  \Delta \zeta$ so that for any $M_1>0$,
\begin{equation}
\label{eq:adjointweightM1}
\left\{
\begin{array}{rl}
- \partial_t \pare{\rho^{M_1}\zeta} -  \Delta \pare{\rho^{M_1}\zeta} = -\left(\rho^{M_1}\right)'\zeta+ \rho^{M_1}g &\mathrm{in}\ (0,T)\times\Omega,\\
\pare{\rho^{M_1}\zeta}= 0&\text{on}\ (0,T)\times\partial\Omega,\\
\pare{\rho^{M_1}\zeta}(T,\cdot)=0& \text{in}\ \Omega.
\end{array}
\right.
\end{equation}
In particular, if we consider $M_1\in (M_0,r m_0)$ then by \eqref{eq:m0M0r0}
and \eqref{eq:comparaisonrhom1m2}
$$
\abs{\left(\rho^{M_1}\right)'} \lesssim \rho^{M_0},\quad \rho^{M_1} \leq \rho^{m_0},
$$
so that
\begin{equation}
\label{eq:rhsL2}
\norme{\left(\rho^{M_1}\right)'\zeta}_{L^2(0,T;L^2(\Omega))}+ \norme{\rho^{M_1}g}_{L^2(0,T;L^2(\Omega))} \lesssim \norme{\rho^{M_0}\zeta}_{L^2(0,T;L^2(\Omega))}+ \norme{\rho^{m_0}g}_{L^2(0,T;L^2(\Omega))}.
\end{equation}
We can apply the maximal regularity result in $L^2$, i.e. \Cref{T03} with $p=2$ to \eqref{eq:adjointweightM1}, 
and use \eqref{eq:rhsL2} and the $L^2$ observability inequality \eqref{eq:ObsL2} to deduce 
\begin{equation}\label{18:23}
\norme{\rho^{M_1}\zeta}_{\mathcal{X}^2} \lesssim 
\norme{\rho^{m_0} g }_{L^2(0,T;L^2(\Omega))}
+ \norme{\rho_0^{m_0} \zeta \chi_{\omega}}_{L^2(0,T;L^2(\Omega))}.
\end{equation}
We then use the Sobolev embedding \eqref{XpLq} to deduce
\begin{equation}\label{22:36}
\mathcal{X}^2 \hookrightarrow L^{q_1}(0,T;L^{q_1}(\Omega))
\end{equation}
with $q_1\geq 2$ defined by
$$
\text{if}\quad \frac{1}{2}-\frac{2}{2+d}\leq \frac{1}{p} \quad \text{then} \ q_1 = p,\quad 
\text{else}\quad \frac{1}{q_1}=\frac{1}{2}-\frac{2}{2+d}.
$$
Then, we consider $M_2\in (M_1,r m_0)$ so that from \eqref{eq:comparaisonrhom1m2},
$$
\abs{\left(\rho^{M_2}\right)'} \lesssim \rho^{M_1},\quad \rho^{M_2} \leq \rho^{m_0},
$$
and with \eqref{22:36} and \eqref{18:23}, we deduce
\begin{align}
\norme{\left(\rho^{M_2}\right)'\zeta}_{L^{q_1}(0,T;L^{q_1}(\Omega))}+ \norme{\rho^{M_2}g}_{L^{q_1}(0,T;L^{q_1}(\Omega))} 
&\lesssim \norme{\rho^{M_1}\zeta}_{\mathcal{X}^2}+ \norme{\rho^{m_0}g}_{L^{q_1}(0,T;L^{q_1}(\Omega))}
\notag
\\
&\lesssim 
\norme{\rho^{m_0} g }_{L^{q_1}(0,T;L^{q_1}(\Omega))}
+ \norme{\rho_0^{m_0} \zeta \chi_{\omega}}_{L^2(0,T;L^2(\Omega))}.
\label{eq:rhsLq1}
\end{align}
Now we apply \cref{T03} to
\begin{equation}
\label{22:34}
\left\{
\begin{array}{rl}
- \partial_t \pare{\rho^{M_2}\zeta} -  \Delta \pare{\rho^{M_2}\zeta} = -\left(\rho^{M_2}\right)'\zeta+ \rho^{M_2}g &\mathrm{in}\ (0,T)\times\Omega,\\
\pare{\rho^{M_2}\zeta}= 0&\text{on}\ (0,T)\times\partial\Omega,\\
\pare{\rho^{M_2}\zeta}(T,\cdot)=0& \text{in}\ \Omega,
\end{array}
\right.
\end{equation}
with $p=q_1$, and using \eqref{eq:rhsLq1}, we obtain
\begin{equation}
\norme{\rho^{M_2}\zeta}_{\mathcal{X}^{q_1}} 
\lesssim 
\norme{\rho^{m_0} g }_{L^{q_1}(0,T;L^{q_1}(\Omega))}
+ \norme{\rho_0^{m_0} \zeta \chi_{\omega}}_{L^2(0,T;L^2(\Omega))}.\label{22:37}
\end{equation}
If $q_1 = p$, then using $H^1(0,T) \hookrightarrow C^0([0,T])$ and $L^{p}(0,T;L^{p}(\Omega)) \hookrightarrow L^2(0,T;L^2(\Omega))$, 
we deduce from the above relation the desired observability inequality \eqref{eq:ObsLp} with $m_1 = M_2$.
Else, we have $q_1 < p$ and we can repeat the argument, that is we use the Sobolev embedding \eqref{XpLq} to deduce
\begin{equation}\label{22:36-bis}
\mathcal{X}^{q_1} \hookrightarrow L^{q_2}(0,T;L^{q_2}(\Omega))
\end{equation}
with $q_2\geq q_1$ defined by
$$
\text{if}\quad \frac{1}{q_1}-\frac{2}{2+d}\leq \frac{1}{p} \quad \text{then} \ q_2 = p,\quad 
\text{else}\quad \frac{1}{q_2}=\frac{1}{q_1}-\frac{2}{2+d}=\frac{1}{2}-2\cdot\frac{2}{2+d}.
$$
Taking $M_3\in (M_2,r m_0)$, and proceeding as above, applying \cref{T03} with $p=q_2$ and using \eqref{22:36-bis} and \eqref{22:37}, we find
$$
\norme{\rho^{M_3}\zeta}_{\mathcal{X}^{q_2}} 
\lesssim 
\norme{\rho^{m_0} g }_{L^{q_2}(0,T;L^{q_2}(\Omega))}
+ \norme{\rho_0^{m_0} \zeta \chi_{\omega}}_{L^2(0,T;L^2(\Omega))}.
$$
We can proceed by induction and since $1/q_n$ decrease by $2/(2+d)$ at each step, after a finite number of steps, we obtain $q_n=p$ and we deduce the desired observability inequality \eqref{eq:ObsLp}.
\end{proof}

\subsection{Controllability of the heat equation with a source term in $L^{p'}$}\label{sec-st}
We use the above observability results to show, by a duality argument, the controllability of a linear system associated with \eqref{eq:heatSL}:
\begin{equation}
\label{eq:heatSource}
\left\{
\begin{array}{rl}
\partial_t y-  \Delta y =  h\chi_{\omega} + F &\mathrm{in}\ (0,T)\times\Omega,\\
y= 0&\mathrm{on}\ (0,T)\times\partial\Omega,\\
y(0,\cdot)=y_0& \mathrm{in}\ \Omega.
\end{array}
\right.
\end{equation}
In order to control the above system, we fix $p\in 2\mathbb{N}^*$ and we consider $m_0$ and $m_1$ as in \cref{L4Obs}.
Then, we introduce
\begin{equation}
\mathcal{Y}_0 := \{ \zeta \in C^{\infty}([0,T]\times\overline{\Omega})\ ;\ \zeta = 0 \ \text{on}\ (0,T)\times \partial \Omega\},
\end{equation}
and we define the following norm for $\zeta \in \mathcal{Y}_0$, 
\begin{equation}
\norme{\zeta}_{\mathcal{Y}} := 
\norme{\rho^{m_0} (\partial_t \zeta +\Delta \zeta)}_{L^p(0,T;L^p(\Omega))}
+\norme{\rho_0^{m_0} \zeta \chi_{\omega}}_{L^p(0,T;L^p(\Omega))}.
\end{equation}
The fact that it is a norm is a consequence of the weighted $L^p$ observability inequality \eqref{eq:ObsLp}. We denote by $\mathcal{Y}$ 
the completion of $\mathcal{Y}_0$ with respect to the norm $\|\cdot\|_{\mathcal{Y}}$.

First, we have the following result that roughly states that a function $\zeta \in \mathcal{Y}$ belongs to some suitable weighted $\mathcal X^p$ spaces.
\begin{Lemma}\label{L1}
Assume $m>m_1$. Then, for any $\zeta\in \mathcal{Y}$,
\begin{equation}\label{15:52}
\left\| \rho_0^{m} \zeta \right\|_{\mathcal{X}^p}  \lesssim \left\| \rho^{m} \zeta \right\|_{\mathcal{X}^p}  \lesssim \norme{\zeta}_{\mathcal{Y}}.
\end{equation}
\end{Lemma}
\begin{proof}
Using $m>m_1$, \eqref{10:09} and \eqref{eq:comparaisonrhom1m2}, we have
\begin{equation}\label{15:37}
\left| \left(\rho^{m}\right)' \right| \lesssim \rho^{m_1},
\quad 
\rho^{m} \lesssim \rho^{m_0}.
\end{equation}
Now, if $\zeta\in \mathcal{Y}$, then 
\begin{equation}
\label{15:38}
\left\{
\begin{array}{rl}
- \partial_t \pare{\rho^{m} \zeta} -  \Delta \pare{\rho^{m} \zeta} 
	= -\left(\rho^{m}\right)' \zeta+ \rho^{m} \left(-\partial_t \zeta-\Delta \zeta\right) &\mathrm{in}\ (0,T)\times\Omega,\\
\pare{\rho^{m}\zeta}= 0&\text{on}\ (0,T)\times\partial\Omega,\\
\pare{\rho^{m}\zeta}(T,\cdot)=0& \text{in}\ \Omega.
\end{array}
\right.
\end{equation}
Combining the observability inequality \eqref{eq:ObsLp} and \eqref{15:37}, we deduce
$$
\left\| -\left(\rho^{m}\right)' \zeta+ \rho^{m} \left(-\partial_t \zeta-\Delta \zeta\right) \right\|_{L^p(0,T;L^p(\Omega))} \lesssim \norme{\zeta}_{\mathcal{Y}}.
$$
Applying the maximal regularity result \cref{T03} on \eqref{15:38} and using the above relation, we deduce the second estimate in \eqref{15:52}. For the first estimate, we use \eqref{eq:comparaisonrho0rho} to obtain that $\norme{\rho_0/\rho}_{W^{1,\infty}(0,T)} \lesssim 1$ and this allows us to conclude the proof.
\end{proof}
%
We now introduce some functional spaces: for $m>0$ and $p\in [1,\infty]$, we set
\begin{equation}
\label{eq:defLpweight}
L^p_{m}(0,T;L^p(\Omega)):= \left\{ f\in L^p(0,T;L^p(\Omega)) \ ; \ \frac{f}{\rho^m}\in L^p(0,T;L^p(\Omega))\right\},
\end{equation}
\begin{equation}
\label{eq:defLpweight2}
L^p_{m,0}(0,T;L^p(\Omega)):= \left\{ f\in L^p(0,T;L^p(\Omega)) \ ; \ \frac{f}{\rho_0^m}\in L^p(0,T;L^p(\Omega))\right\},
\end{equation}
endowed with the following norm
\begin{equation}
\label{eq:defnormLpweight}
\left\| f \right\|_{L^p_{m}(0,T;L^p(\Omega))} := \left\| \frac{f}{\rho^m}\right\|_{L^p(0,T;L^p(\Omega))},
\quad
\left\| f \right\|_{L^p_{m,0}(0,T;L^p(\Omega))} := \left\| \frac{f}{\rho_0^m}\right\|_{L^p(0,T;L^p(\Omega))}.
\end{equation}

Let us consider, for any $y_0 \in L^{p'}(\Omega)$ and $F \in L^{p'}_{m_1}(0,T;L^{p'}(\Omega))$, the functional $J=J_{y_0,F}$ defined as follows:
\begin{multline}
\label{eq:defJ}
J(\zeta) := \frac{1}{p} \iint_{(0,T)\times \Omega} \rho^{m_0 p} (-\partial_t \zeta - \Delta \zeta)^p dt dx 
+ \frac{1}{p} \iint_{(0,T)\times \Omega} \rho_0^{m_0 p} \zeta^p \chi_\omega^p dt dx 
\\
- \iint_{(0,T)\times \Omega} F \zeta dt dx 
- \int_{\Omega} y_0(x) \zeta(0,x) dx.
\end{multline}
Using the $L^p$ observability inequality \eqref{eq:ObsLp}, we can check that $J\in C^1(\mathcal{Y};\mathbb{R})$ is a strictly convex
and coercive functional on $\mathcal{Y}$. In particular, $J$ admits a unique minimum $\overline{\zeta}$.
We can thus define, for $y_0 \in L^{p'}(\Omega)$ and $F \in L^{p'}_{m_1}(0,T;L^{p'}(\Omega))$, the following maps
\begin{equation}\label{15:45}
\mathcal{M}_1(y_0,F):=\overline{\zeta}, 
\quad
\mathcal{M}_2(y_0,F):=\rho^{m_0 p} (-\partial_t \overline{\zeta} - \Delta \overline{\zeta})^{p-1},
\quad
\mathcal{M}_3(y_0,F):=-\rho_0^{m_0 p} \chi_\omega^{p-1} \overline{\zeta}^{p-1}.
\end{equation}
\begin{Proposition}\label{P1}
Assume $p\in 2\mathbb{N}^*$ and $r\in (1,p')$ and let us consider $m_0$ and $m_1$ given by \cref{L4Obs}.
For any $y_0 \in L^{p'}(\Omega)$ and $F \in L^{p'}_{m_1}(0,T;L^{p'}(\Omega))$, let us set
\begin{equation}\label{yh}
\overline{\zeta}=\mathcal{M}_1(y_0,F), \quad
y=\mathcal{M}_2(y_0,F), \quad
h=\mathcal{M}_3(y_0,F). 
\end{equation}
\begin{enumerate}
\item \textbf{Existence of a solution.} 
We have $y \in L^{p'}_{m_0}(0,T;L^{p'}(\Omega))$ and $h\in L^{p'}_{m_0,0}(0,T;L^{p'}(\Omega))$, together with the estimates
\begin{equation}\label{14:58}
\left\| \overline{\zeta} \right\|_{\mathcal{Y}}^p \lesssim 
\left\| F\right\|_{L^{p'}_{m_1}(0,T;L^{p'}(\Omega))}^{p'}
+\left\| y_0 \right\|_{L^{p'}(\Omega)}^{p'},
\end{equation}
\begin{equation}
\label{eq:estimateL43}
\norme{y}_{L^{p'}_{m_0}(0,T;L^{p'}(\Omega))} 
+ \norme{h}_{L^{p'}_{m_0,0}(0,T;L^{p'}(\Omega))} 
\lesssim 
\norme{F}_{L^{p'}_{m_1}(0,T;L^{p'}(\Omega))} + \norme{y_0}_{L^{p'}(\Omega)}.
\end{equation}
Moreover, $y$ is the very weak solution of \eqref{eq:heatSource} associated with $F$, $h$ and $y_0$ in the sense of \cref{D3}.
\item \textbf{Odd behavior of the control.} The control $h$ satisfies $h^{1/(p-1)} \in \mathcal{X}^p$ 
and 
\begin{equation}\label{10:37}
h^{1/(p-1)}(0,\cdot)=h^{1/(p-1)}(T,\cdot)=0 \quad \text{in} \ \Omega,
\end{equation}
together with the estimate
\begin{equation}\label{h1/3}
\norme{h^{1/(p-1)}}_{\mathcal{X}^p} \lesssim 
\norme{F}_{L^{p'}_{m_1}(0,T;L^{p'}(\Omega))}^{1/(p-1)} + \norme{y_0}_{L^{p'}(\Omega)}^{1/(p-1)}.
\end{equation}
\item \textbf{Regularity of the solution.} Assume that $y_0\in W^{2/p,p'}(\Omega)$ and that $y_0=0$ on $\partial\Omega$ if $p=2$.
Then for any $m<m_0$, $y/\rho^{m}\in \mathcal{X}^{p'}$ together with the estimate
\begin{equation}
\label{eq:RegEstimate}
\norme{ \frac{y}{\rho^{m}} }_{\mathcal{X}^{p'}} 
\lesssim 
\norme{F}_{L^{p'}_{m_1}(0,T;L^{p'}(\Omega))} 
+ \norme{y_0}_{W^{2/p,p'}(\Omega)}.
\end{equation}
In particular, $y(T,\cdot) = 0$.
\end{enumerate}
\end{Proposition}

The first point will be obtained from Euler-Lagrange equation. The odd behavior of the control, i.e. \eqref{h1/3}, remarking that $p-1$ is odd, comes from the identification of $h$ in \eqref{yh}, \eqref{15:45} and from  a weighted $\mathcal X^p$ estimate of $\overline{\zeta}$. 
Finally, the regularity result on the solution comes from a maximal parabolic regularity result. Note that if $p\neq 2$, then $p\geq 4$ and $p'<3/2$ so that we do not need to impose the compatibility condition $y_0=0$ on $\partial\Omega$.
\begin{proof}[Proof of \cref{P1}]
We start by writing the Euler-Lagrange equation for $J$ at $\overline{\zeta}$ to obtain
\begin{multline}
\label{eq:EL}
\iint_{(0,T)\times \Omega} \rho^{m_0 p} (-\partial_t \overline{\zeta} - \Delta \overline{\zeta})^{p-1} (-\partial_t \zeta - \Delta \zeta) dt dx 
+ \iint_{(0,T)\times \Omega} \rho_0^{m_0 p} \chi_\omega^p \overline{\zeta}^{p-1} \zeta  dt dx 
\\
= \iint_{(0,T)\times \Omega} F \zeta dt dx 
+ \int_{\Omega} y_0(x) \zeta(0,x) dx \quad (\zeta\in \mathcal{Y}).
\end{multline}
Taking $\zeta=\overline{\zeta}$ in the above relation and using Young's inequality and the $L^p$ observability inequality \eqref{eq:ObsLp}, we deduce \eqref{14:58}.

Then, \eqref{yh} and \eqref{15:45} yield
$$
\left| \frac{y}{\rho^{m_0}} \right|^{p'}= \rho^{m_0 p} \left| \partial_t \overline{\zeta} + \Delta \overline{\zeta}\right|^{p},
\quad 
\left| \frac{h}{\rho_0^{m_0}} \right|^{p'}= \left| \rho_0^{m_0} \chi_\omega \overline{\zeta}\right|^{p},
$$
and we deduce \eqref{eq:estimateL43} from \eqref{14:58}.

Moreover, \eqref{eq:EL} and \eqref{yh} imply
\begin{multline}
\label{eq:EL-2}
\iint_{(0,T)\times \Omega} y (-\partial_t \zeta - \Delta \zeta) dt dx 
=\iint_{(0,T)\times \Omega} \chi_\omega h  \zeta  dt dx 
+ \iint_{(0,T)\times \Omega} F \zeta dt dx 
\\
+ \int_{\Omega} y_0(x) \zeta(0,x) dx \quad (\zeta\in C^\infty_c([0,T)\times \Omega)),
\end{multline}
that is $y$ is the very weak solution to \eqref{eq:heatSource} associated with the control $h$, $F$ and $y_0$ in the sense of \cref{D3}.

For the second point, from \eqref{yh} and \eqref{15:45}, we have
\begin{equation}
h^{1/(p-1)}= -\rho_0^{p' m_0} \overline{\zeta} \chi_{\omega},
\end{equation}
and since $p'm_0>m_1$, we can apply \cref{L1} with $m=p'm_0$ and 
we deduce \eqref{h1/3} from \eqref{14:58} and \eqref{15:52}.
Since $p'm_0>m_1$, there exists $r>0$ such that $p'm_0-r > m_1$, we then obtain \eqref{10:37} because
$$ \norme{\rho_0^{-r} h^{1/(p-1)}}_{C([0,T];W^{2/p',p}(\Omega))} \lesssim \norme{\rho_0^{-r} h^{1/(p-1)}}_{\mathcal{X}^p} \lesssim \norme{\rho_0^{p'm_0 -r} \overline{\zeta}}_{\mathcal{X}^p} \lesssim \norme{\zeta}_{\mathcal{Y}}.$$

Finally, for the last point, we write the system satisfied by $y/\rho^{m}$:
\begin{equation}
\label{eq:equationyrho}
\left\{
\begin{array}{rl}
\partial_t \left(\frac{y}{\rho^{m}}\right)-  \Delta \left(\frac{y}{\rho^{m}}\right) 
	=  \frac{h}{\rho^{m}}\chi_{\omega} 
	+ \frac{F}{\rho^{m}} 
-m \frac{\rho'}{\rho^{m+1}} y
&\mathrm{in}\ (0,T)\times\Omega,\\[4mm]
\frac{y}{\rho^{m}}= 0&\mathrm{on}\ (0,T)\times\partial\Omega,\\[3mm]
\frac{y}{\rho^{m}}(0,\cdot)=\frac{y_0}{\rho^{m}(0)}& \mathrm{in}\ \Omega.
\end{array}
\right.
\end{equation}
By using $m<m_0<m_1$, \eqref{eq:comparaisonrho0rho}, \eqref{eq:comparaisonrhom1m2} and \eqref{eq:estimateL43}, we have
$$
\norme{
\frac{h}{\rho^{m}}\chi_{\omega} + \frac{F}{\rho^{m}} -m \frac{\rho'}{\rho^{m+1}} y
}_{L^{p'}(0,T;L^{p'}(\Omega))}
\lesssim \norme{F}_{L^{p'}_{m_1}(0,T;L^{p'}(\Omega))} + \norme{y_0}_{L^{p'}(\Omega)}
$$
%
Applying \cref{T03} to \eqref{eq:equationyrho} with the above estimate,
we deduce the regularity estimate on $y$, i.e. \eqref{eq:RegEstimate}.
\end{proof}

\subsection{$L^{\infty}$ bound on the control and $L^q$ estimate of the nonlinearity}

From now on, we assume $r\in (1,p')$ and we assume that $m_0$ and $m_1$ are given by \cref{L4Obs} with this $r$. In particular
they satisfy \eqref{10:09} which yields
$$
0<m_0 p -m_1(p-1)<m_0.
$$
First we have the following result on the control $h$.
\begin{Lemma}\label{L7}
Assume $p$ satisfies \eqref{eq:defpOdd} and \eqref{eq:defpgrand}. 
Then for any 
\begin{equation}\label{20:22}
0 \leq m < m_0 p -m_1(p-1),
\end{equation}
the control $h$ given by \eqref{yh} satisfies $h^{1/(2n+1)}\in \mathcal{X}^p$ and $h\in L_{m,0}^{\infty}(0,T;L^\infty(\Omega))$ with the estimate
\begin{equation}\label{16:51}
\left\| h \right\|_{L_{m,0}^{\infty}(0,T;L^\infty(\Omega))} 
+\norme{\left(\frac{h}{\rho_0^{m} }\right)^{1/(2n+1)} }_{\mathcal X^p}^{2n+1}
\lesssim
\left\| F\right\|_{L^{p'}_{m_1}(0,T;L^{p'}(\Omega))}
+\left\| y_0 \right\|_{L^{p'}(\Omega)}.
\end{equation}
\end{Lemma}
In the above result, $p$ has to be sufficiently large to get that $\mathcal{X}^p$ is an algebra, and this enables us to get that $h^{1/(2n+1)}$ is sufficiently smooth, because $p-1 = (2n+1)(2k+1)$, as expected in \eqref{eq:oddcontrol}.
\begin{proof}[Proof of \cref{L7}]
Since $p$ satisfies \eqref{eq:defpgrand}, we can apply \cref{L01} and deduce that $\mathcal{X}^p$ is an algebra. 
On the other hand, from \cref{P1}, $h^{1/(2n+1)}\in \mathcal{X}^p$, we can thus conclude by using that
$$h^{1/(2n+1)}=\left(h^{1/(p-1)}\right)^{2k+1}.$$ 

Now, from \eqref{yh} and \eqref{15:45}, we can write
\begin{equation}
\label{eq:intermediateequalityhLemmaLinfty}
\frac{h}{\rho_0^{m}}=- \left(\rho_0^{(m_0 p-m)/(p-1)} \chi_\omega \overline{\zeta}\right)^{p-1}.
\end{equation}

If $m$ satisfies \eqref{20:22}, then $(m_0 p-m)/(p-1) > m_1$, we can apply \cref{L1} and use \eqref{eq:intermediateequalityhLemmaLinfty}, 
\eqref{eq:comparaisonrho0rho}, \eqref{14:58} to obtain
\begin{multline*}
\norme{\frac{h}{\rho_0^{m}}}_{\mathcal X^p} 
\lesssim \norme{\left(\rho_0^{(m_0 p-m)/(p-1)} \chi_\omega \overline{\zeta}\right)^{p-1}}_{\mathcal X^p} 
\lesssim \norme{\rho_0^{(m_0 p-m)/(p-1)} \chi_\omega \overline{\zeta}}_{\mathcal X^p}^{p-1}\\
\lesssim \norme{\overline{\zeta}}_{\mathcal Y}^{p-1} \lesssim  \left\| F\right\|_{L^{p'}_{m_1}(0,T;L^{p'}(\Omega))}
+\left\| y_0 \right\|_{L^{p'}(\Omega)}.
\end{multline*}
We obtain \eqref{16:51} by using that $\mathcal{X}^p$ is an algebra and \eqref{XpLq}.
\end{proof}

\begin{Proposition}\label{P7}
Let $N\in \mathbb{N}^*$, $N\geq 2$ and assume $p,q$ satisfying $q\geq p'$, \eqref{11:44}, \eqref{eq:defpOdd} and \eqref{eq:defpgrand}. 
Let us consider $m_0$ and $m_1$ given by \cref{L4Obs} with 
\begin{equation}\label{15:51}
r:=\frac{p}{p-1+\frac{1}{N}}\in (1,p').
\end{equation}
For any $y_0\in W^{\frac{2}{q'},q}_0(\Omega)$ and $F\in L^q_{m_1}(0,T;L^q(\Omega))$, 
and for any $m$ satisfying \eqref{20:22}, $y$ defined by \eqref{yh}
satisfies $y/\rho^{m} \in \mathcal X^q$ and $y^N\in L^q_{m_1}(0,T;L^q(\Omega))$ with the estimates 
\begin{equation}\label{16:59-Result}
\norme{ \frac{y}{\rho^{m}} }_{\mathcal{X}^{q}} 
\lesssim 
\norme{F}_{L^{q}_{m_1}(0,T;L^{q}(\Omega))} 
+ \norme{y_0}_{W^{2/q',q}(\Omega)},
\end{equation}
\begin{equation}\label{17:05}
\left\| y^N \right\|_{L^q_{m_1}(0,T;L^q(\Omega))}
\lesssim 
\left(
\norme{F}_{L^{q}_{m_1}(0,T;L^{q}(\Omega))} 
+ \norme{y_0}_{W^{2/q',q}(\Omega)}\right)^N.
\end{equation}
\end{Proposition}
The goal of the above result is to get an appropriate $L^q$ bound on the nonlinearity, this would be a first step in order to prove the local null-controllability of the semi-linear heat equation.
\begin{proof}
We define $q_1$ as follows
\begin{equation}\label{22:28}
\text{if}\quad \frac{1}{q}\leq \frac{1}{p'}-\frac{2}{d+2}, \ \text{then} \ q_1 = q, \  \text{else}
\ \frac{1}{q_1}= \frac{1}{p'}-\frac{2}{d+2}.
\end{equation}
In both cases, we have $q\geq q_1$ and $1/q'\geq 1/q_1'$. 

We deduce from \eqref{eq:RegEstimate} and the Sobolev embedding \eqref{XpLq} that for any
$\widetilde{m}<m_0$,
\begin{equation}\label{16:45-a}
\norme{ y }_{L^{q_1}_{\widetilde{m}}(0,T;L^{q_1}(\Omega))} 
\lesssim 
\norme{F}_{L^{p'}_{m_1}(0,T;L^{p'}(\Omega))} 
+ \norme{y_0}_{W^{2/p,p'}(\Omega)}.
\end{equation}
We then consider $m$ satisfying \eqref{20:22}. We have in particular $m<m_0<m_1$ and we can write 
\begin{equation*}
\left\{
\begin{array}{rl}
\partial_t \left(\frac{y}{\rho^{m}}\right)-  \Delta \left(\frac{y}{\rho^{m}}\right) 
	=  \frac{h}{\rho^{m}}\chi_{\omega} 
	+ \frac{F}{\rho^{m}} 
-m \frac{\rho'}{\rho^{m+1}} y
&\mathrm{in}\ (0,T)\times\Omega,\\[4mm]
\frac{y}{\rho^{m}}= 0&\mathrm{on}\ (0,T)\times\partial\Omega,\\[3mm]
\frac{y}{\rho^{m}}(0,\cdot)=\frac{y_0}{\rho^{m}(0)}& \mathrm{in}\ \Omega.
\end{array}
\right.
\end{equation*}
Applying \cref{T03} on the above equation and using \eqref{16:51} and \eqref{16:45-a} with $\tilde{m} \in (m,m_0)$ together with \eqref{eq:comparaisonrho0rho}, \eqref{eq:comparaisonrhom1m2}, we deduce
\begin{multline}\label{16:55}
\norme{ \frac{y}{\rho^{m}} }_{\mathcal{X}^{q_1}} 
\lesssim \left\| h \right\|_{L_{m}^{\infty}(0,T;L^\infty(\Omega))} + 
\norme{F}_{L^{q_1}_{m_1}(0,T;L^{q_1}(\Omega))} + \norme{\frac{\rho'}{\rho^{m+1}}y}_{L^{q_1}(0,T;L^{q_1}(\Omega))}
\\
\lesssim \norme{F}_{L^{q_1}_{m_1}(0,T;L^{q_1}(\Omega))}
+ \norme{y_0}_{W^{2/q_1',q_1}(\Omega)} + \norme{ y }_{L^{q_1}_{\widetilde{m}}(0,T;L^{q_1}(\Omega))} 
\\
\lesssim \norme{F}_{L^{q_1}_{m_1}(0,T;L^{q_1}(\Omega))} + \norme{y_0}_{W^{2/q',q}(\Omega)}.
\end{multline}
We can proceed by induction, using again \eqref{XpLq},
and since the corresponding sequence $1/q_n$ decreases by $2/(d+2)$ (see \eqref{22:28}) at each step, we obtain
after a finite number of steps that for any $m$ satisfying \eqref{20:22}, we have 
\begin{equation}\label{16:59}
\norme{ \frac{y}{\rho^{m}} }_{\mathcal{X}^{q}} 
\lesssim 
\norme{F}_{L^{q}_{m_1}(0,T;L^{q}(\Omega))} 
+ \norme{y_0}_{W^{2/q',q}(\Omega)}.
\end{equation}
Using that $q$ satisfies \eqref{11:44} so that the Sobolev embedding \eqref{eq:embeddingXqLnq} holds, we deduce that
\begin{equation}\label{17:02}
\norme{ \frac{y^N}{\rho^{N m}} }_{L^{q}((0,T)\times \Omega)} 
\lesssim 
\left(
\norme{F}_{L^{q}_{m_1}(0,T;L^{q}(\Omega))} 
+ \norme{y_0}_{W^{2/q',q}(\Omega)}\right)^N.
\end{equation}
From \eqref{10:09} and \eqref{15:51}, we have
\begin{equation*}\label{09:44}
\frac{m_1}{N}<m_0 p -m_1(p-1)
\end{equation*}
so that we can take $m=m_1/N$ in \eqref{16:59}, \eqref{17:02} and we deduce \eqref{17:05}.
\end{proof}

\subsection{A Schauder fixed-point argument}
\label{Sec:schauder}
Let us consider the hypotheses of \cref{P7} and assume $y_0\in W_0^{2/q',q}(\Omega)$. Then, using the conclusion of
\cref{P7}, we can define the mapping
\begin{equation}\label{09:58}
\mathcal{N} : L^{q}_{m_1}(0,T;L^{q}(\Omega)) \to L^{q}_{m_1}(0,T;L^{q}(\Omega)), \quad F\mapsto y^N,
\end{equation}
where $y=\mathcal{M}_2(y_0,F)$.
Moreover, using  \eqref{17:05}, we deduce that
if $R_0:=\norme{y_0}_{W^{2/q',q}(\Omega)}$ is small enough, then the closed set
\begin{equation}\label{10:00}
B_{R_0}:= \left\{ F\in L^{q}_{m_1}(0,T;L^{q}(\Omega))  \ ; \ \left\| F \right\|_{L^q_{m_1}(0,T;L^q(\Omega))}\leq R_0 \right\}
\end{equation}
is invariant by $\mathcal{N}$.

\begin{Proposition}\label{L3}
The mapping $\mathcal{N} : B_{R_0} \to B_{R_0}$ defined above is continuous and $\mathcal{N}(B_{R_0})$ is relatively compact into $B_{R_0}$.
\end{Proposition}

\begin{proof}
Let us consider a sequence $\left( F_n\right)_n$ of $B_{R_0}$. We write $y_n=\mathcal{M}_2(y_0,F_n)$.
Then we can use \eqref{16:59-Result} to obtain that 
$\left(y_n/\rho^{(m_1/N)} \right)_n$ is bounded in $\mathcal{X}^{q}$. Applying \cref{L01}, we deduce that, up to a subsequence, 
$$
\frac{y_n}{\rho^{(m_1/N)}} \to \frac{y}{\rho^{(m_1/N)}} \quad \text{in} \ L^{qN}((0,T)\times \Omega), 
$$
for some $y\in L^{qN}_{m_1/N}(0,T; L^{qN}(\Omega))$.
We deduce that $\mathcal{N}(B_{R_0})$ is relatively compact into $B_{R_0}$.

To show the continuity of $\mathcal{N}$, we consider $F_1,F_2\in B_{R_0}$ and we write (see \eqref{15:45} and \eqref{yh}) for $i=1,2$,
$$
\overline{\zeta}_i := \mathcal{M}_1(y_0,F_i), 
\quad
y_i:=\mathcal{M}_2(y_0,F_i),
\quad
h_i:=\mathcal{M}_3(y_0,F_i).
$$
From the Euler-Lagrange equation \eqref{eq:EL} for $J_{y_0,F_1}$ and $J_{y_0,F_2}$, we deduce
\begin{multline}
\label{eq:EL2}
\iint_{(0,T)\times \Omega} \rho^{m_0 p} 
\left[  (-\partial_t \overline{\zeta}_1 - \Delta \overline{\zeta}_1)^{p-1}
-(-\partial_t \overline{\zeta}_2 - \Delta \overline{\zeta}_2)^{p-1} \right](-\partial_t \zeta - \Delta \zeta) dt dx 
\\
+ \iint_{(0,T)\times \Omega} \rho_0^{m_0 p} \chi_\omega^p \left( \overline{\zeta}_1^{p-1}-\overline{\zeta}_2^{p-1}\right) \zeta  dt dx 
= \iint_{(0,T)\times \Omega} \left( F_1-F_2\right) \zeta dt dx 
\quad (\zeta\in \mathcal{Y}).
\end{multline}
In the above relation, we take $\zeta=\overline{\zeta}_1-\overline{\zeta}_2$ in the above relation and we combine it with the observability inequality 
\eqref{eq:ObsLp} and with the relation
$$
(x_1-x_2)^p\lesssim (x_1^{p-1}-x_2^{p-1})(x_1-x_2) \quad (x_1,x_2\in \mathbb{R}),
$$
to deduce
\begin{equation}\label{14:58-bis}
\left\| \overline{\zeta}_1-\overline{\zeta}_2 \right\|_{\mathcal{Y}}^p \lesssim 
\left\| F_1-F_2\right\|_{L^{p'}_{m_1}(0,T;L^{p'}(\Omega))}^{p'}.
\end{equation}
Moreover, using that
$$
\left| x_1^{p-1}-x_2^{p-1}\right| \lesssim 
\left| x_1-x_2\right| (\left|x_1\right|^{p-2} + \left|x_2\right|^{p-2}) \quad (x_1,x_2\in \mathbb{R}),
$$
we obtain from \eqref{eq:comparaisonrho0rho}
\begin{multline*}
\left| \frac{h_1-h_2}{\rho^{m}} \right|
=\left(\frac{\rho_0}{\rho}\right)^{m} \left| 
 \left(\rho^{(m_0 p-m)/(p-1)} \chi_\omega \overline{\zeta}_1\right)^{p-1}
 -\left(\rho^{(m_0 p-m)/(p-1)} \chi_\omega \overline{\zeta}_2\right)^{p-1}
 \right|
 \\
\lesssim
\left| \rho^{(m_0 p-m)/(p-1)}  \left( \overline{\zeta}_1-\overline{\zeta}_2\right) \right|
\left( \left(\rho^{(m_0 p-m)/(p-1)}  \overline{\zeta}_1\right)^{p-2}
 +\left(\rho^{(m_0 p-m)/(p-1)}  \overline{\zeta}_2\right)^{p-2}
 \right).
\end{multline*}
Thus, if $m$ satisfies \eqref{20:22}, the above relation combined with \eqref{eq:comparaisonrho0rho}, \eqref{eq:defpgrand} that guarantees that $\mathcal X^p$ is an algebra and \cref{L1} yields
\begin{multline*}
\left\| h_1-h_2 \right\|_{L^\infty_{m}(0,T;L^\infty(\Omega))}
\lesssim
\left\| \rho^{(m_0 p-m)/(p-1)}  \left( \overline{\zeta}_1-\overline{\zeta}_2\right)\right\|_{\mathcal{X}^p}
\left( \left\| \rho^{(m_0 p-m)/(p-1)}  \overline{\zeta}_1 \right\|_{\mathcal{X}^p}^{p-2}
 +\left\| \rho^{(m_0 p-m)/(p-1)}  \overline{\zeta}_2 \right\|_{\mathcal{X}^p}^{p-2}
 \right)
\\
\lesssim
\left\| \overline{\zeta}_1-\overline{\zeta}_2 \right\|_{\mathcal{Y}}
\left( \left\|  \overline{\zeta}_1 \right\|_{\mathcal{Y}}^{p-2}
 +\left\|  \overline{\zeta}_2 \right\|_{\mathcal{Y}}^{p-2}
 \right).
\end{multline*}
Therefore, using \eqref{14:58} and \eqref{14:58-bis}, we find
\begin{multline}\label{18:15}
\left\| h_1-h_2 \right\|_{L^\infty_{m}(0,T;L^\infty(\Omega))}
\\
\lesssim
\left\| F_1-F_2\right\|_{L^{p'}_{m_1}(0,T;L^{p'}(\Omega))}^{1/(p-1)}
\left( 
\left\| F_1\right\|_{L^{p'}_{m_1}(0,T;L^{p'}(\Omega))}
+\left\| F_2\right\|_{L^{p'}_{m_1}(0,T;L^{p'}(\Omega))}
+\left\| y_0 \right\|_{L^{p'}(\Omega)}
\right)^{(p-2)/(p-1)}.
\end{multline}
Note that $y_1 - y_2$ satisfies the following system
\begin{equation*}
\left\{
\begin{array}{rl}
\partial_t \left(\frac{y_1 - y_2}{\rho^{m}}\right)-  \Delta \left(\frac{y_1 - y_2}{\rho^{m}}\right) 
	=  \frac{h_1-h_2}{\rho^{m}}\chi_{\omega} 
	+ \frac{F_1-F_2}{\rho^{m}} 
-m \frac{\rho'}{\rho^{m+1}} (y_1 - y_2)
&\mathrm{in}\ (0,T)\times\Omega,\\[4mm]
\frac{y}{\rho^{m}}= 0&\mathrm{on}\ (0,T)\times\partial\Omega,\\[3mm]
\frac{y}{\rho^{m}}(0,\cdot)=\frac{y_0}{\rho^{m}(0)}& \mathrm{in}\ \Omega.
\end{array}
\right.
\end{equation*}
Now, we follow the same proof as in \cref{P7} and we use that $m=m_1/N$
satisfies \eqref{20:22}
to deduce from \eqref{18:15} that
\begin{equation}\label{16:59-bis}
\norme{ \frac{y_1-y_2}{\rho^{m_1/N}} }_{\mathcal{X}^{q}} 
\lesssim 
\norme{F_1-F_2}_{L^{q}_{m_1}(0,T;L^{q}(\Omega))} 
\\
+R_0^{(p-2)/(p-1)} \left\| F_1-F_2\right\|_{L^{p'}_{m_1}(0,T;L^{p'}(\Omega))}^{1/(p-1)}.
\end{equation}
We then write
$$
\left| \frac{y_1^N-y_2^N}{\rho^{m_1}} \right|
\lesssim
\frac{\left| y_1-y_2 \right|}{\rho^{m_1/N}}
\frac{\left| y_1\right|^{N-1}+\left| y_2\right|^{N-1}}{\rho^{m_1 (N-1)/N}}
$$
so that from Hölder's inequality, we have
$$
\left\| y_1^N-y_2^N \right\|_{L^q_{m_1}(0,T;L^q(\Omega))}
\lesssim
\left\|  y_1-y_2 \right\|_{L^{qN}_{m_1/N}(0,T;L^{qN}(\Omega))}
\left(
\left\|  y_1\right\|_{L^{qN}_{m_1/N}(0,T;L^{qN}(\Omega))}^{N-1}
+\left\|  y_1\right\|_{L^{qN}_{m_1/N}(0,T;L^{qN}(\Omega))}^{N-1}
\right).
$$
Combining this relation with the Sobolev embedding \eqref{XpLq}, \eqref{16:59-Result}, \eqref{16:59-bis}, we deduce that 
\begin{multline}
\label{eq:proofholdercontinuitynonlinear}
\left\| y_1^N-y_2^N \right\|_{L^q_{m_1}(0,T;L^q(\Omega))}
\lesssim
\left( \norme{F_1-F_2}_{L^{q}_{m_1}(0,T;L^{q}(\Omega))} 
+R_0^{(p-2)/(p-1)} \left\| F_1-F_2\right\|_{L^{p'}_{m_1}(0,T;L^{p'}(\Omega))}^{1/(p-1)}
\right)
\\
\times
\left(
\norme{F_1}_{L^{q}_{m_1}(0,T;L^{q}(\Omega))} 
+\norme{F_2}_{L^{q}_{m_1}(0,T;L^{q}(\Omega))} 
+ \norme{y_0}_{W^{2/q',q}(\Omega)}
\right)^{N-1}
\end{multline}
which implies the continuity of $\mathcal{N}$.
\end{proof}

\begin{Remark}
In the above proof, let us remark that we show that the mapping $\mathcal{N} : B_{R_0} \to B_{R_0}$ is $\alpha$-H\"older continuous with $\alpha = 1/(p-1)$
(see \eqref{eq:proofholdercontinuitynonlinear}). It is not clear if this mapping is Lipschitz continuous or if we can show that for $R_0$ small enough it is contractive.
As a consequence, in the proof of \cref{th:MainresultSmooth}, we do not apply the Banach fixed-point theorem (as it can be done with the method proposed in \cite{LTT13})
and we use instead the Schauder fixed-point theorem.
\end{Remark}

We are now in a position to prove \cref{th:MainresultSmooth}.
\begin{proof}[Proof of \cref{th:MainresultSmooth}]
From \cref{L3}, if $R_0:=\norme{y_0}_{W^{2/q',q}(\Omega)}$ is small enough, then 
the mapping $\mathcal{N} : B_{R_0} \to B_{R_0}$ defined by \eqref{09:58} is continuous, where
$B_{R_0}$ is the closed convex set defined by \eqref{10:00}. 
Moreover, $\mathcal{N}(B_{R_0})$ is relatively compact in $B_{R_0}$
and we can thus apply the Schauder fixed point theorem to deduce the existence of a fixed point $F\in B_{R_0}$. 
Setting $y=\mathcal{M}_2(y_0,F)$ and $h=\mathcal{M}_3(y_0,F)$, we can apply \cref{P1}, \cref{L7} and \cref{P7} and obtain
that $h$ satisfies \eqref{eq:oddcontrol}, that $y$ is the strong solution of \eqref{eq:heatSL} associating with $h$ and $y_0$
and that for any $m$ satisfying \eqref{20:22},
$y/\rho^{m} \in \mathcal X^q$,  $y^N\in L^q_{m_1}(0,T;L^q(\Omega))$,
$h^{1/(2n+1)}\in \mathcal{X}^p$, $h \in L_{m,0}^{\infty}(0,T;L^\infty(\Omega))$ together with the estimates 
\eqref{16:59-Result-2}.
\end{proof}

\section{Proof of \cref{th:Mainresult1}}
\label{sec:ProofMainresults}

The goal of this part is to prove the local null-controllability of \eqref{eq:ReactionDiffusionSL}. 
\begin{proof}[Proof of \Cref{th:Mainresult1}]
As explained in the introduction, the proof is divided into two steps. 
\paragraph{Step 1: control of the first equation in $(0,T/2)$.}
First we apply \cref{Th:WPLinftySL}: there exists $\widetilde \delta>0$ small enough such that if
\begin{equation}
\label{09:27}
\norme{y_{2,0}}_{L^{\infty}(\Omega)}\leq \widetilde\delta, \quad
\norme{g}_{L^\infty((0,T/2)\times \Omega)} \leq \widetilde\delta,
\end{equation}
the system 
\begin{equation}\label{eq:heat-y2}
\left\{
\begin{array}{rl}
\partial_t y_2- \Delta y_2=y_2^{N_3} + g &\mathrm{in}\ (0,T/2)\times\Omega,\\
y_2= 0&\mathrm{on}\ (0,T/2)\times\partial\Omega,\\
y_2(0,\cdot)=y_{2,0} & \mathrm{in}\ \Omega,
\end{array}
\right.
\end{equation}
admits a unique weak solution in the sense of \cref{D1}. 
Now we  apply \cref{th:mainresult2} to
\begin{equation}\label{eq:heat-y1}
\left\{
\begin{array}{rl}
\partial_t y_1- \Delta y_1=y_1^{N_1} + h \chi_{\omega} &\mathrm{in}\ (0,T/2)\times\Omega,\\
y_1= 0&\mathrm{on}\ (0,T/2)\times\partial\Omega,\\
y_1(0,\cdot)=y_{1,0} & \mathrm{in}\ \Omega.
\end{array}
\right.
\end{equation}
There exists $\delta>0$ such that for any $y_{1,0}\in L^\infty(\Omega)$ with
\begin{equation}\label{20:39}
\left\| y_{1,0} \right\|_{L^\infty(\Omega)} \leq \delta,
\end{equation}
there exists a control $h\in L^\infty(0,T/2;L^\infty(\Omega))$ such that $y_1(T/2,\cdot)=0$ and 
$$
\left\| y_1 \right\|_{L^\infty(0,T/2;L^\infty(\Omega))} \lesssim \left\| y_{1,0} \right\|_{L^\infty(\Omega)}.
$$
Assuming \eqref{eq:smalldatasystem} with $\delta>0$ possibly smaller, we have that 
$$
g:=y_1^{N_2} \in L^\infty((0,T)\times \Omega)
$$
satisfies \eqref{09:27} so that we have obtained at this step a control 
$h\in L^\infty(0,T/2;L^\infty(\Omega))$, such that 
\eqref{eq:ReactionDiffusionSL} admits a weak solution $(y_1,y_2)$ in $(0,T/2)$ and $y_1(T/2,\cdot)=0$.
By using \cref{L20}, $y_{2,T/2}:=y_{2}(T/2,\cdot)$ satisfies
$$
\left\| y_{2,T/2} \right\|_{L^{\infty}(\Omega)} \lesssim \delta.
$$ 
\paragraph{Step 2: control of the second equation in $(T/2,T)$ through a fictitious odd control.}
By taking $\delta>0$ possibly smaller, we can apply \cref{th:mainresult2} to
\begin{equation}\label{eq:heat-y2-ter}
\left\{
\begin{array}{rl}
\partial_t y_2- \Delta y_2=H\chi_{\omega} + y_2^{N_3} &\mathrm{in}\ (T/2,T)\times\Omega,\\
y_2= 0&\mathrm{on}\ (T/2,T)\times\partial\Omega,\\
y_2(T/2,\cdot)=y_{2,T/2} & \mathrm{in}\ \Omega.
\end{array}
\right.
\end{equation}
We deduce the existence of a control $H$ such that $y_2(T,\cdot)=0$ and such that
$$
H^{1/N_2}\in L^p(T/2,T;W^{2,p}(\Omega)) \cap W^{1,p}(T/2,T;L^p(\Omega)), \quad
H^{1/N_2}(T/2, \cdot)=H^{1/N_2}(T,\cdot)=0.
$$
We then set, in $(T/2,T)$,
$$
y_1:=\left(H\chi_{\omega}\right)^{1/N_2}, \quad 
h:=\partial_t y_1-\Delta y_1-y_1^{N_1}\in L^p((T/2,T)\times \Omega).
$$
Concatenating $y_1$, $y_2$ and $h$ between the two steps, we can check that $h\in L^p((0,T)\times \Omega)$, 
that $(y_1,y_2)$ is the weak solution of \eqref{eq:ReactionDiffusionSL} and that \eqref{eq:NullStateu1u2} holds.
This concludes the proof of \Cref{th:Mainresult1}.
\end{proof}

\bibliographystyle{alpha}
\bibliography{ref-LocalOdd}

\end{document}